\documentclass[12pt]{amsart}
\usepackage{a4wide}
\usepackage{amssymb}
\usepackage{amsthm}
\usepackage{amsmath}
\usepackage{amscd}
\usepackage{verbatim}
\usepackage{upgreek}
\usepackage[all]{xy}
\usepackage{hyperref}
\usepackage{mathrsfs}
\numberwithin{equation}{section}

\theoremstyle{plain}
\newtheorem{theorem}{Theorem}[section]
\newtheorem{corollary}[theorem]{Corollary}
\newtheorem{lemma}[theorem]{Lemma}
\newtheorem{proposition}[theorem]{Proposition}

\theoremstyle{definition}
\newtheorem{definition}[theorem]{Definition}
\newtheorem{remark}[theorem]{Remark}

\theoremstyle{remark}

\newcommand{\A}{\mathbb{A}}
\newcommand{\R}{\mathbb{R}}

\newcommand{\Q}{\mathbb{Q}}
\newcommand{\Z}{\mathbb{Z}}

\newcommand{\C}{\mathbb{C}}

\newcommand{\stab}{\operatorname{stab}}
\newcommand{\Hom}{\operatorname{Hom}}

\newcommand{\Aut}{\operatorname{Aut}}

\newcommand{\spec}{\operatorname{spec}}

\newcommand{\Ker}{\operatorname{Ker}}

\newcommand{\Img}{\operatorname{Im}}

\newcommand{\Cent}{\operatorname{Cent}}

\newcommand{\vol}{\operatorname{vol}}

\newcommand{\der}{\operatorname{der}}

\newcommand{\tr}{\operatorname{tr}}

\newcommand{\unip}{\operatorname{unip}}
\newcommand{\ssc}{\operatorname{sc}}

\newcommand{\Gal}{\operatorname{Gal}}

\newcommand{\reg}{\operatorname{reg}}

\newcommand{\cusp}{\operatorname{cusp}}

\newcommand{\disc}{\operatorname{disc}}
\newcommand{\cont}{\operatorname{cont}}
\newcommand{\temp}{\operatorname{temp}}
\newcommand{\fin}{\operatorname{fin}}
\newcommand{\geo}{\operatorname{geo}}
\newcommand{\elll}{\operatorname{ell}}
\newcommand{\Int}{\operatorname{Int}}
\newcommand{\dimm}{\operatorname{dim}}
\newcommand{\eexp}{\operatorname{exp}}
\newcommand{\image}{\operatorname{Image}}
\newcommand{\ram}{\operatorname{ram}}
\newcommand{\Out}{\operatorname{Out}}
\newcommand{\sss}{\operatorname{ss}}
\newcommand{\ac}{\operatorname{ac}}
\newcommand{\de}{\operatorname{det}}
\begin{document}

\title[ Multiplicity formula and stable trace formula ]{Multiplicity formula and stable trace formula }
\author[]{}
\dedicatory{}

\address{}
\email{}

\thanks{}

\date{\today}
\author[Peng Zhifeng ]{ Peng Zhifeng}
\address{ Department of Mathematics, National University of Singapore}
\email{matpeng@nus.edu.sg}

\maketitle

\begin{abstract}
Let $G$ be a connected reductive group over $\mathbb{Q}$. In this paper, we give the stabilization of the local trace formula. In particular, we construct the explicit form of the spectral side of the stable local trace formula in the Archimedean case, when one component of the test function is cuspidal. We also give the multiplicity formula for discrete series. At the same time, we obtain the stable version of $L^{2}$-Lefschetz number formula.

\end{abstract}

\maketitle

%\tableofcontents

\section{Introduction}
\label{sect: introduction}
  Suppose that $G$ is a connected reductive group over $\mathbb{Q}$, and $\Gamma$ is an arithmetic subgroup of $G(\mathbb{R})$ defined by congruence conditions. Consider the regular representation $R$ with $G(\mathbb{R})$ acting on $L^{2}(\Gamma \backslash G(\mathbb{R}))$ through the right translation. The fundamental problem is to decompose $R$ into a direct sum of irreducible representations. In general, we decompose $R$ into two parts

                            \[R=R_{\disc}\oplus R_{\cont},\]
where $R_{\disc}$ is the sum of discrete series, and $R_{\cont}$ is the continuous spectrum. The continuous spectrum can be understood by Eisenstein series, which was studied by Langlands \cite{L3}. It suffices to study $R_{\disc}$. If $\pi_{\R}\in R_{\disc}$ is an irreducible representation, we denote $R_{\disc}(\pi_{\R})$ for the $\pi_{\R}$-isotypical subspace of $R_{\disc}$. Then

                    \[R_{\disc}(\pi_{\R})=\pi_{\R}^{\oplus m_{\disc}(\pi_{\R})},\]
where $m_{\disc}(\pi_{\R})$ is the multiplicity. A classical problem is to find a finite summation formula for $m_{\disc}(\pi_{\R})$.

\bigskip
   If $\pi_{\R}$ belongs to the square integrable discrete series, and $\Gamma\backslash G(\R)$ is compact, then Langlands \cite{L1} gave a formula for $m_{\disc}(\pi_{\mathbb{R}})$. If $\Gamma\backslash G(\R)$ is noncompact, the first result is for $G(\R)=SL_{2}(\mathbb{R})$, the formula for $m_{\disc}(\pi_{\mathbb{R}})$ appeared in Selberg's paper \cite{S}. For $G$ having $\mathbb{R}$-rank one, there is a formula for it in \cite{OW}. In general, for $G$ having any $\R$-rank, Arthur \cite{A3} studied the sum of multiplicities

    \begin{equation}
    \sum_{\pi_{\R}\in\Pi_{\disc}(\mu)}m_{\disc}(\pi_{\mathbb{R}},K_{0})
    \end{equation}
    by using the invariant trace formula, where the $L$-packet $\Pi_{\disc}(\mu)$ is a finite set of discrete series representations of $G(\R)$ with the same infinitesimal character $\mu$, and $K_{0}$ is an open compact subgroup of the finite adelic group $G(\mathbb{A}_{\fin})$. More generally, we can consider Hecke operators $h$ on $L^{2}(\Gamma\backslash G)$ that commute with $R$, and write $R_{\disc}(\pi_{\R},h)$ for the restriction of $h$ to $R_{\disc}(\pi_{\R})$. Arthur \cite{A3} obtained a formula for

    \begin{equation}\label{eq:sumt}
    \sum_{\pi_{\R}\in \Pi_{\disc}(\mu)}\tr(R_{\disc}(\pi_{\R},h)),
    \end{equation}
    under a weak regularity condition on $\mu$. The spectral side of invariant trace formula corresponds to \eqref{eq:sumt}, if the test function is taken to be a stable cuspidal function $f_{\mu}$ associated to $\mu$. Therefore, the explicit formula for \eqref{eq:sumt} follows from the geometric side of invariant trace formula. A key point is that the invariant distribution $I_{M}(\gamma,f_{\mu})$ vanishes, if $\gamma$ is not semisimple.

\bigskip
We shall give a formula for the multiplicity of single representation $m_{\disc}(\pi_{\R})$ and \[\tr(R_{\disc}(\pi_{\R},h)),\] which was conjectured by Spallone and Wakatsuki \cite[Conjecture 1]{Sp}, who also had checked two special cases. When one tries to use the invariant trace formula to obtain the multiplicity formula, with the test function being the pseudo-coefficient $f_{\pi_{\R}}$ for a single representation $\pi_{\R}$, the invariant distribution $I_{M}(\gamma,f_{\pi_{\R}})$ in general does not vanish for $\gamma$ having nontrivial unipotent part. Now we do not know how to obtain an explicit formula for $I_{M}(\gamma,f_{\pi_{\R}})$ in this case. To capture a single representation and obtain the stable cuspidal function, we need to use the endoscopy theory and the stable trace formula.

\bigskip
  Fortunately, when $G$ is a $K$-group, Arthur \cite{A9}, \cite{A10}, \cite{A11} has obtained the stabilization of the general trace formula in 2003, assuming the Fundamental Lemma and the the weighted Fundamental Lemma. In 2008, Ngo \cite{N} proved the Fundamental Lemma, and the weighted Fundamental Lemma was proved by Chaudouard-Laumon. As a result, we have an unconditional stabilization of global and local formal trace formula. From this we obtain the following proposition by combining with the splitting formula.
  \begin{proposition}
For any $h\in\mathcal{H}(G(\A_{\fin})),$  and $\pi_{\R}\in \Pi_{\disc}(\mu),$ with the infinitesimal character $\mu$ being regular, we have
\begin{equation}
\begin{split}
\tr(R_{\disc}(\pi_{\R},h))&=I(f_{\pi_{\R}}h) \\
&=\sum_{G^{'}\in\mathcal{E}_{\elll}(G)}\iota(G,G^{'})\sum_{M^{'}\in\mathcal{L}^{G^{'}}}|W^{M^{'}}_{0}||W^{G^{'}}_{0}|^{-1}\sum_{\delta\in\Delta(M^{\prime},V,\zeta)}b^{M^{\prime}}(\delta)S^{G^{'}}_{M^{'}}(\delta,f_{\pi_{\R}})(h_{M})^{M^{'}}(\delta)
\end{split}
\end{equation}
\end{proposition}

\bigskip
 See section \ref{sec:gs} for more details. The stable coefficients $b^{M^{\prime}}(\delta)$ and the stable distributions $S^{G^{\prime}}_{M^{\prime}}(\delta,f_{\pi_{\R}})$ are not explicit. However, the transfer function $f^{G^{\prime}}_{\pi_{\R}}$ is a stable cuspidal function for endoscopic group $G^{\prime}$ of $G$, we obtain the vanishing of  $S^{G^{\prime}}_{M^{\prime}}(\delta,f_{\pi_{\R}})$, if $\delta$ is not semisimple. It remains to study this distribution in the case of semisimple elements of $G^{\prime}(\mathbb{R})$. We shall give the explicit formula for the stable distribution $S^{G^{\prime}}_{M^{\prime}}(\delta,f_{\pi_{\R}})$ through comparison of the stable local trace formula with the stable Weyl integral formula. We need to stabilize directly the spectral side of the local trace formula in the Archimedean place.

\bigskip
 In the $p$-adic case, when the test function is cuspidal at two places, Arthur \cite{A7} gave a concrete formula about the geometric side of stable local trace formula, which is just an inner product. We shall establish a formula in the Archimedean case, with the test function being cuspidal at only one place, so it is more complicated. It suffices to study the stable distribution on quasisplit groups. Our work relies on the harmonic analysis on reductive groups of Harish-Chandra \cite{Harish-Chandra1}, \cite{Harish-Chandra2}, \cite{Harish-Chandra3}, and on the Langlands \cite{L2} classification of the irreducible representations of real algebraic groups. We also need the work of Shelstad \cite{S3}, \cite{S4}, who classified the tempered representations, directly constructed the spectral transfer factors, and gave the inverse adjoint relations. For the stable local trace formula, we need to build the transfer factors with respect to the virtual tempered representations $\tau=(M,\pi,r)$ and the Langlands parameters. The main obstruction is to show that the representation theoretic $R$-groups $R_{\pi}$ and the endoscopic $R$-groups are compatible. This is the essential part that allows us to stabilize the local trace formula. We will also build the general inverse adjoint relations, and obtain an explicit formula for the spectral side of the stable local trace formula, c.f. Section $6$.

 \bigskip
\begin{theorem}
If $f = f_1 \times \bar{f_2}$, $f_1 \in C_{\cusp}(G(\R),\zeta)$, and $f_2 \in C(G(\R),\zeta)$, then
\begin{align*}
I_{\disc}(f) =&\int_{T_{\elll}(G,\zeta)} i^{G}(\tau)f_{1,G}(\tau)\overline{f_{2,G}(\tau)} d\tau  \\
=&\sum_{G'\in\mathcal{E}_{\elll}(G)} \iota(G, G') \widehat{S}^{G'}(f'),
\end{align*}
 and $\widehat{S}^{G'}(f')$ is a stable distribution on $G^{\prime}$,  where
 \begin{align*} i^{G}(\tau) =& |d(\tau)|^{-1}|R_{\pi,r}|^{-1}, \\ \widehat{S}^{G'}(f') =& \int_{\Phi_2(G', \zeta)} S^{G'}(\phi')\widetilde{f_1}'(\phi')\overline{f_{2}'(\phi')} d\phi', \\
S^{G'}(\phi^{\prime}) =& \frac{1}{|\mathcal{S}_{\phi^{\prime}}|}, \qquad \phi = \xi'\circ\phi'.
\end{align*}
\end{theorem}

\bigskip
 We shall obtain a stable local trace formula $S_{\geo}=S_{\spec}$ in Theorem \ref{theorem,ls} and an explicit formula for the stable distribution. We then have the following main theorem.

 \bigskip
 \begin{theorem}
 If $h\in \mathcal{H}(G(\A_{\fin}))$, and the highest weight of representation $\mu$ is regular, then we have

  \begin{equation*}
  \begin{split}
  &\tr(R_{\disc}(\pi_{\R},h))   \\
=&\sum_{G^{\prime}\in\mathcal{E}_{\elll}(G)}\iota(G,G^{\prime})\sum_{M^{\prime}\in\mathcal{L}^{G^{\prime}}}(-1)^{\dimm(A_{M^{\prime}}/A_{G^{\prime}})}|W^{M^{\prime}}_{0}||W^{G^{\prime}}_{0}|^{-1}\sum_{\delta\in\{M^{\prime}(\Q)\}}P_{\mu}(M^{\prime})S\Phi_{M^{\prime}}(\phi^{\prime}_{\mu},\delta)(h_{M})^{M^{\prime}}(\delta),
\end{split}
\end{equation*}
 and the multiplicity formula of the discrete series is

 \[  m_{\disc}(\pi_{\R},K_{0})=\tr(R_{\disc}(\pi_{\R},\mathit{1}_{K_{0}})).
   \]
\end{theorem}

\bigskip
We will also give a stable $L^{2}$-Lefschetz formula, when the test function is stable cuspidal. This case is also studied in an unpublished work of Kottwitz \cite{K4}.

\begin{theorem}
  For any $h\in \mathcal{H}(G(\mathbb{A}_{\fin}))$, we have
  \[\mathcal{L}_{\mu}(h)=\sum_{G^{'}\in\mathcal{E}_{\elll}(G)}\iota(G,G^{'})\sum_{M^{'}\in\mathcal{L}^{G^{'}}}|W^{M^{'}}_{0}||W^{G^{'}}_{0}|^{-1}\sum_{\delta\in\{M^{\prime}(\mathbb{Q})\}}F_{\mu}(M^{\prime})S\Phi_{M^{\prime}}(\phi^{\prime},\delta)(h_{M})^{M^{\prime}}(\delta).\]
\end{theorem}

\bigskip
 The content of the paper is as follows. In section \ref{sec:pd}, we will introduce $K$-groups, which are unions of connected reductive groups. If $F$ is a $p$-adic field, then a $K$-group over $F$ is still a connected reductive group. However, if $F$ is an Archimedean field, then a $K$-group over $F$ is not connected. Any connected reductive group $G_{1}$ is a component of a unique $K$-group $G$. The invariant and the stable distributions can be extended to $K$-group as in \cite{A8}.

\bigskip
In section \ref{sec:mi}, we will obtain the relation between multiplicity and invariant trace formula. However when the test function is a pseudo-coefficient, we cannot give an explicit invariant trace formula. This is because we cannot cancel the contributions coming from the unipotent elements. For test functions which are not stable, we need to stabilize the invariant trace formula to overcome this obstruction. In general, the invariant trace formula is the identity obtained from two different expansions of a certain linear form $I(f)$. One side is the geometric expansion

\[I(f)=\sum_{M\in\mathcal{L}}|W^{M}_{0}||W^{G}_{0}|^{-1}\sum_{\gamma\in\Gamma(M,S)}a^{M}(S,\gamma)I_{M}(\gamma,f), \]
which is a linear combination of distributions parameterized by conjugacy classes $\gamma$
in Levi subgroups $M(F_{S})$. The other side is the spectral expansion

\[I(f)=\sum_{M\in\mathcal{L}}|W^{M}_{0}||W^{G}_{0}|^{-1}\int_{\Pi(M,S)}a^{M}(S,\pi)I_{M}(\pi,f)\,d\pi,\]
which is a linear combination of distributions parameterized by representations $\pi$ of Levi subgroup $M(F_{S})$.
  Here $ f \in\mathcal{H}(G,S)$ in the Hecke algebra of $G(F_{S})$ (see \cite{A1}, \cite{A2}), $S$ is the finite set of place in $F$. Arthur \cite{A11} had stabilized the invariant trace formula. So we can express the multiplicity by using the geometric side of the stable global trace formula.

  \bigskip
  In section \ref{sec:gs}, We apply the splitting formula to reduce the local component of the global trace formula for the Archimedean case. Compared with the local trace formula, the local component of the global trace formula is more complicated, because it contains contributions from nontrivial unipotent elements. When the test function is stable cuspidal, those contributions vanish, and the local components of the local trace formula coincide with the local components of the global trace formula. Moreover, the pseudo-coefficient of a representation can be transferred to a stable cuspidal function under the Shelstad transfer mapping. So it is enough to study the stabilization of local trace formula. In general, the geometric side of the local trace formula concerns semisimple regular elements. The spectral side of the local trace formula concerns the tempered representations and is thus much simpler in comparison with the global trace formula.

\bigskip
 The spectral side of the invariant local trace formula contains a natural object, which is called the virtual character. A simple invariant local trace formula can be given in terms of the virtual representation as in \cite{A5}. In section \ref{sec:tc}, we introduce the virtual representation for the Archimedean case, and define the transfer factors $\Delta(\tau,\phi)$ together with the inverse factors $\Delta(\phi,\tau)$, building on Shelstad's work \cite{S3}. Then we stabilize the spectral side of the local trace formula. When one component of the test function is cuspidal, we just need to consider elliptic representations. In section \ref{sec:se}, we obtain a formula for the spectral side of the stable local trace formula.

\bigskip
In section \ref{sec:cs}, we study the main term $S^{G}_{M}(\delta,\phi)$. To do this, we need to stabilize the Weyl integral formula, which connects between the geometric and spectral sides of a local component of the local trace formula. We can then compare with the stable local trace formula, and the main term of the formula will appear.

\bigskip
In section \ref{sec:mf}, we establish the relation between $S^{G}_{M}(\delta,f)$ and the invariant main term $\Phi_{M}(\gamma,f)$. This allows us to overcome the key obstacle. We have $S^{G}_{M}(\delta,f_{\phi})=0$, if $\delta$ is not semisimple, where $f_{\phi}$ is a stable cuspidal function which arises by a transfer from a pseudo-coefficient of $\pi$.
We shall collect various terms, and they will be combined into our main formula in Theorem \ref{theorem: m}. In section \ref{sec:sl}, we shall give a stable formula for $L^{2}$-Lefschetz number, when the test function is stable cuspidal.

\smallskip
\noindent\textbf{Acknowledgements.} I would like to thank my advisor James Arthur's help and encouragement with this project. This paper is indebted to Chung Pang Mok for much useful conversation. I would like to thank Bin Xu for helpful conversation. We gratefully acknowledge generous support provided by the National Natural Science Foundation of P.R. China Grant No.11601503.

\bigskip
\section{Preliminaries and distributions of K-groups } \label{sec:pd}
\subsection{Notation}

 Let $G$ be a reductive group over $\mathbb{Q}$, $M$ be a Levi subgroup of $G$, $P$ be a parabolic subgroup, and $A_{G}$ be the $\mathbb{Q}$-split component of the center of $G$. If $\gamma$ is a semisimple element of $G$, we denote by $G({\mathbb{Q},\gamma})$ the centralizer of $\gamma$ in $G(\mathbb{Q})$, and write $G_{\gamma}$ for the identity component of $G(\mathbb{Q},\gamma)$. Write $G_{\der}$ for the derived group of $G$, $G_{\ssc}$ for the simply connected cover of $G_{\der}$. We say that $G$ is cuspidal, if $G(\mathbb{R})$ contains a maximal $\mathbb{R}$-torus $T$ such that $T/A_{G}$ is anisotropic over $\mathbb{R}$. In other words, the $\mathbb{Q}$-split component coincide with $\mathbb{R}$-split component $A_{G(\mathbb{R})}$ and the real group $G(\mathbb{R})$ contains an $\mathbb{R}$-elliptic maximal torus. A torus $T$ in $G$ is elliptic if $T/A_{G}$ is anisotropic. An element of $G(\mathbb{R})$ is elliptic if its centralizer in $G(\mathbb{R})$ is an elliptic torus of $G$. Let $X(G)_{\mathbb{R}}$ be the module of $\mathbb{R}$-rational character on $G$, and $\mathfrak{a}_{G}=\Hom(X(G)_{\mathbb{R}},\mathbb{R})$.

\bigskip
We denote by $\mathbb{A}$ the ring of ad\'eles of $\mathbb{Q}$, and denote by $\mathbb{A}_{\fin}$ the finite part of ad\'eles ring over $\mathbb{Q}$, so that $\mathbb{A}=\mathbb{R}\times \mathbb{A}_{\fin}$. We denote by $G^{\prime}$ an endoscopy group of $G$. The center of $G$ is denoted as $Z(G)$, and we will denote by $Z$ a central induced torus in $G$ over $\mathbb{Q}$. Throughout the paper, $F$ will be a field of characteristic 0.

\bigskip
\subsection{$K$-groups}

 $K$-groups are natural objects for studying the stabilization of general trace formula, which contain several connected components. To work with several groups simultaneously is suitable for studying the transfer properties of the various objects in the trace formula. The use of several inner forms is due originally to Vogan. When Kottwitz learned of Vogan's idea, he applied it to the Langlands-Shelstad transfer factors. We shall follow Arthur's discussion \cite{A8}, where he extended the geometric transfer factors to $K$-groups.

\bigskip
\begin{definition}
$G$ is called a $K$-group over a local field $F$, if
\begin{enumerate}
\item $G$ is an algebraic variety whose connected components are reductive algebraic groups over $F$, endowed with an equivalence class of objects $\{(\psi,u)\}$. Here $(\psi,u)=\{ (\psi_{\alpha\beta},u_{\alpha\beta}): \alpha, \beta\in \pi_{0}(G)\}$,
$\psi_{\alpha\beta}:G_{\beta}
 \to G_{\alpha}$  is an isomorphism over $\bar{\mathnormal{F}}$, and $u_{\alpha\beta}: \Gamma \to G_{\alpha,\ssc}$ is an 1-cocycle, where $\Gamma =\Gal(\bar{\mathnormal{F}}/\mathnormal{F})$. We require that $\{(\psi,u)\}$
satisfy the compatibility conditions,\label{compatibility}

(i) \quad $\psi_{\alpha\beta} \tau(\psi_{\alpha\beta})^{-1}=\Int(u_{\alpha\beta}(\tau));$

(ii) \quad $ \psi_{\alpha\gamma}=\psi_{\alpha\beta}\psi_{\beta\gamma};$

(iii)\quad $u_{\alpha\gamma}(\tau)=\psi_{\alpha\beta,\ssc}(u_{\beta\gamma}(\tau))u_{\alpha\beta}(\tau),$
\smallskip
for any $\alpha,\beta,\gamma \in \pi_{0}(G)$ and $\tau\in \Gamma$,
\item the corresponding sequence
$$\{1\} \longrightarrow  \{u_{\alpha\beta}:\beta\in \pi_{0}(G)\} \longrightarrow \mathnormal{H}^{1}(F,G_{\alpha}) \xrightarrow{K_{G_{\alpha}}}\pi_{0}(Z(\widehat{G})^{\Gamma})^{\ast}.$$
   of pointed sets is exact. Here $\alpha\in \pi_{0}(G)$, the map $K_{G_{\alpha}}$ is defined in \cite[\S 1]{K2}.
\end{enumerate}
 \end{definition}
\bigskip
The notation $\pi_{0}(G)$ is a set of indices for the components of $G$, and we write $\pi_{0}(Z(\widehat{G})^{\Gamma})$ for the set of connected components of $Z(\widehat{G})^{\Gamma}$ as usual.

 We say that two such families $(\psi,u)$ and $(\psi^{'},u^{'})$ are equivalent, if there are elements $g_{\alpha\beta}\in G_{\alpha,\ssc}$ such that $\psi^{'}_{\alpha\beta}=\Int(g_{\alpha\beta})\psi_{\alpha\beta}$ and $u^{'}_{\alpha\beta}(\tau)=g_{\alpha\beta}u_{\alpha\beta}(\tau)\tau(g_{\alpha\beta})^{-1}$, for any $\alpha,\beta\in\pi_{0}(G)$ and $\tau\in\Gamma$. We call a representative $( \psi, u)$ from the equivalence
 class as a frame for $G$. If $F$ is $p$-adic, $K_{G_{\alpha}}$ is a bijection \cite[Theorem 1.2]{K2}, so a $K$-group is just a connected reductive group. If $F$ is Archimedean, the kernel of $K_{G_{\alpha}}$ is the image of $\mathnormal{H}^{1}(F,G_{\alpha,\ssc})$ in $\mathnormal{H}^{1}(F,G_{\alpha})$ \cite[Theorem 1.2]{K2}, and the number of components of a $K$-group over $\mathbb{R}$ therefore is equal to the number of classes in this image.

\bigskip
Suppose that $G$ is a $K$-group, then we can write $G=\coprod_{\alpha\in\pi_{0}(G)}G_{\alpha}$. A homomorphism between $K$-groups $G$ and $\bar{G}$ over $F$ is a morphism

   \[  \theta= \coprod_{\alpha}(\theta_{\alpha}:G_{\alpha}\to \bar{G}_{\bar{\alpha}})\]

   from $G$ to $\bar{G}$ (as varieties over $F$) that preserves all the structures. In other words, it satisfies the following two properties:

   \begin{enumerate}
   \item  For any $\alpha\in\pi_{0}(G)$, and $\bar{\alpha}=\theta(\alpha)$ the image of $\alpha$ in $\pi_{0}(G)$, the restriction $\theta_{\alpha}:G_{\alpha}\rightarrow\bar{G}_{\bar{\alpha}}$ is a homomorphism of connected algebraic groups.
   \item  There are frames $(\psi,u)$ and $(\bar{\psi},\bar{u})$ for $G$ and $\bar{G}$, such that $\theta_{\alpha}\circ\psi_{\alpha\beta}=\bar{\psi}_{\bar{\alpha}\bar{\beta}}\circ\theta_{\beta}$, and $\bar{u}_{\bar{\alpha}\bar{\beta}}=\theta_{\alpha,\ssc}(u_{\alpha\beta})$, for each $\alpha,\beta\in \pi_{0}(G)$.
   \end{enumerate}

\bigskip
     An isomorphism of $K$-groups is an invertible homomorphism. Arthur introduces the notion of weak isomorphism in \cite[\S4]{A9}. It satisfies all the requirements of an isomorphism except for the condition relating $u_{\bar{\alpha}\bar{\beta}}$ with $u_{\alpha\beta}$, so that one can identify $K$-groups that differ only by the choice of functions $\{u_{\alpha\beta}\}$. If we are given a connected reductive group $G_{1}$ over $F$, we can find a $K$-group $G$ over $F$, such that $G_{\alpha_{1}}=G_{1}$ for some $\alpha_{1}\in \pi_{0}(G)$. There could be several such $G$, but the weak isomorphism class of $G$ is uniquely determined by $G_{1}$. In particular, any connected quasisplit group $G^{\ast}$ has a quasisplit inner $K$-form $G$, which is unique up to weak isomorphism. We say that $K$-group $G$ is quasisplit if it has a connected component that is quasisplit over $F$.

\bigskip
     The Levi subgroups $M$ of a $K$-group $G$ was defined \cite[\S 1]{A8}. For any such $M$, we construct the associated objects $W(M)$, $\mathcal{P}(M)$, $\mathcal{L}(M)$ and $\mathcal{F}(M)$ as in \cite[\S 1]{A8}, which represent the Weyl group, the set of the parabolic subgroups for which the component of Levi subgroup equals $M$, the set of Levi subgroups which contain the Levi subgroup $M$ and the set of parabolic subgroups for which the component of the Levi subgroups contain the Levi subgroup $M$ respectively. They play the same role as in the connected case. We can also form a dual group $\widehat{G}$ for $G$, and a dual Levi subgroup $\widehat{M}\subset\widehat{G}$ for $M$. For any such group $\widehat{M}$, we also have the analogue object $\mathcal{P}(\widehat{M})$, $\mathcal{L}(\widehat{M})$ and $\mathcal{F}(\widehat{M})$, with the understanding that the sets contain only $\Gamma$-stable elements. It comes with a bijection $L\to \widehat{L}$ from $\mathcal{L}(M)$ to $\mathcal{L}(\widehat{M})$, and a bijection $P\to \widehat{P}$ from $P(M)$ to $P(\widehat{M})$.

\bigskip
 Invariant harmonic analysis for connected real groups extends in a natural way to $K$-groups. For example, we have the Harish-Chandra's Schwartz space,

  \[ C(G)=\bigoplus_{\alpha\in\pi_{0}(G)} C(G_{\alpha})\]

on $G(\mathbb{R})$, and its invariant analogue

 \[ I(G)= \bigoplus_{\alpha\in \pi_{0}(G_{\alpha})}I(G_{\alpha}).\]

  Elements in $C(G)$ are the functions on $G(\mathbb{R})$, and elements in $I(G)$ can be regarded as the functions on the disjoint union

\[ \Pi_{\temp}(G)=\coprod_{\alpha\in\pi_{0}(G)}\Pi_{\temp}(G_{\alpha})\]
of sets of irreducible tempered representations on the groups $G_{\alpha}(\R)$, or as functions on the disjoint union

\[\Gamma _{\reg}(G)=\coprod_{\alpha\in\pi_{0}(G)}\Gamma _{\reg}(G_\alpha)\]
of the sets of strongly regular conjugacy classes in the groups $G_{\alpha}(\mathbb{R})$.

\bigskip
   For purpose of induction argument, it is convenient to fix a central character datum $(Z,\zeta)$ for $G$, where $Z$ is an induced torus over $\mathbb{R}$, with the central embedding $Z\to Z_{\alpha}\subset{G_{\alpha}}$ that is compatible with isomorphisms $\psi_{\alpha\beta}$. The second component $\zeta$ is a character on $Z(\mathbb{R})$, which corresponds to a character $\zeta_{\alpha}$ on $Z_{\alpha}(\mathbb{R})$ for each $\alpha$.

 We can then form the space
\[C(G,\zeta)= \bigoplus_{\alpha\in\pi_{0}(G)}C(G_{\alpha},\zeta_{\alpha})\]
of $\zeta^{-1}$ equivariant Schwartz functions on $G(\mathbb{R})$, and its invariant analogue
\[ I(G,\zeta)=\bigoplus_{\alpha\in\pi_{0}(G)}I(G_{\alpha},\zeta_{\alpha}).\]
 Elements in $I(G,\zeta)$ may be regarded either as $\zeta^{-1}$-equivariant functions on $\Pi_{\temp}(G,\zeta)$ or on $\Gamma_{\reg}(G/Z)$.

\bigskip
 If $\gamma$ lies in $\Gamma(G_{\alpha})$, we write $G_{\gamma}$ for the centralizer in $G_{\alpha}$ of (some representative of) $\gamma$. Two classes $\gamma_{1}$ and $\gamma_{2}$ in $\Gamma(G)$ with $\gamma_{i}\in\Gamma(G_{\alpha_{i}})$ for $ i=1,2$
 are stably conjugate, if $\psi_{\alpha_{1}\alpha_{2}}(\gamma_{2})$ is conjugate in $G_{\alpha_{1}}(\bar{F})$ to $\gamma_{1}$, for any frame $(\psi,u)$. We can then write $\Delta_{\reg}(G(F))$ for the set of strongly regular stable conjugacy classes in $G(F)$. There is a canonical injection $\delta\to \delta^{\ast}$ from $\Delta_{\reg}(G)$ to the set $\Delta_{\reg}(G^{\ast})=\Delta_{\reg}(G^{\ast}(F))$ of strongly regular stable conjugacy classes in the quasisplit inner twist $G^{\ast}(F)$.

\bigskip
An endoscopic datum for $G$ is defined entirely in terms of the dual group $\widehat{G}$, and is therefore no different from the connected case. $\mathcal{E}(G)$ will stand for the set of isomorphism classes of endoscopic data for $G$ that are relevant to $G$. An element in $\mathcal{E}(G)$ is therefore the image of some elliptic endoscopic datum $M^{'}=(M^{'},\mathcal{M}^{'},s^{'},\xi^{'})$ in $\mathcal{E}_{\elll}(M)$, for a Levi subgroup $M$ of $G$ and a dual Levi subgroup $\widehat{M}$ of $\widehat{G}$. Here the elliptic datum means that the image of $\mathcal{M}^{'}$ in ${}^{L}M$ is
contained in no proper parabolic subgroup ${}^{L}M$,
or equivalently that $(Z(\mathcal{M'})^{\Gamma})^{0}=(Z(\widehat{M})^{\Gamma})^{0}$. The set $\mathcal{\mathcal{E}}(G)$ embeds into the larger set $\mathcal{\mathcal{E}}(G^{\ast})$, which we identify as the collection of all isomorphism classes of endoscopic data for $G$. For each $G^{\prime}=(G^{\prime},\mathcal{G}^{\prime},s,\xi^{\prime})\in \mathcal{\mathcal{E}}(G^{\ast})$, we fix a central extension as in \cite[\S 2]{A7}.
$$1 \longrightarrow  \widetilde{Z}^{\prime} \longrightarrow \widetilde{G}^{\prime} \longrightarrow G^{\prime}\longrightarrow 1$$
of $G'$ by a central induced torus of $\widetilde {Z}^{\prime}$. Then there exists an $L$-morphism $\widetilde{\xi'}:\mathcal{G}^{\prime}\to {}^{L}\widetilde{G}^{\prime}$.

\bigskip
A $K$-group is a natural domain for the transfer factors of \cite{LS}. If $F$ is Archimedean, Arthur extends the transfer factors to the $K$-groups. It is known \cite{S4} that the set of conjugacy classes in the stable conjugacy classes of $K$-group can be parametrized by the set $\mathcal{E}(T)=\Img(\mathnormal{H}^{1}(\Gamma,T_{\ssc})\to\mathnormal{H}^{1}(\Gamma,T))$, where $T$ is the maximal torus of $G$. When $G$ is a connected group, then it's parametrized by a subset $\mathscr{D}(T)$ of $\mathcal{\mathcal{E}}(T)$. Here $\mathscr{D}(T)=\Ker(\mathnormal{H}^{1}(\Gamma,T)\to\mathnormal{H}^{1}(\Gamma,G))$, where $T$ is a compact maximal torus. Similarly, an $L$-packet of discrete series representations is parametrized by $\mathcal{E}(T)$.

\bigskip
  For example, If we consider the connected group $G^{'}=G^{'}_{\ssc}=SU(2,1)$ over $\mathbb{R}$, its $K$-group is $G=G^{'}\coprod G_{1}$, where $G_{1}=SU(3)$. Then a stable conjugacy class of regular elliptic elements of $G{'}$ consists of three conjugacy classes parametrized by three of the four elements of $\mathnormal{H}^{1}(\Gamma,T)$. Similarly, an L-packet of discrete series representations of $G^{\prime}$ is parametrized by three elements of the same group, and we can obtain the other conjugacy class and the other representation from $G_{1}$.

\bigskip
 We also extend Langlands-Shelstad transfer mapping to $K$-groups.
  \[ \varphi: \mathcal{H}(G,\zeta) \to SI(\widetilde{G}^{'},\widetilde{\zeta}^{'})\]
\[\varphi(f)=f^{'}(\delta^{'})=\sum_{\gamma\in\Gamma_{\reg}(G,\zeta)}\Delta(\delta^{'},\gamma)f_{G}(\gamma).\]
  Here $\mathcal{H}(G,\zeta)$ is the Hecke algebra. $SI(\widetilde{G}^{'},\widetilde{\zeta^{'}})$ is the space of stable orbital integrals of functions.
If $F$ is non-archimedean, this is the main result of Waldspurger \cite{W2} and Ngo \cite{N}. If $F$ is Archimedean, the result was proved in Shelstad's paper \cite{S1}.

\bigskip
If $F$ is a global field, then there is a notion of $K$-group over $F$. Such a $G$ satisfies the global analogue of the property as above, together with a local product structure. A local product structure on $G$ is a family of local $K$-groups $(G_{v},F_{v})$, indexed by the valuations of $F$, and a family of homomorphisms $G\to G_{v}$ over $F_{v}$ whose restricted direct product $G(\A)\to\prod_{v}G_{v}(F_{v})$ is an isomorphism over $\mathbb{A}$. Such a structure determines a surjective map
\[\alpha \mapsto \alpha_{V}=\prod_{v\in V}\alpha_{v}, \qquad   \alpha\in \pi_{0}(G),\alpha_{v}\in \pi_{0}(G_{v}),\]
of components. We also have a group theoretic injection of $G_{\alpha}(F)$ to $G_{\alpha_{v}}(F_{v})$ for each $\alpha\in \pi_{0}(G)$. We shall write \[G_{V}(F_{V})=
 \prod_{v\in V}G_{v}(F_{v})=\prod_{v\in V}\coprod_{\alpha_{v}}G_{v,\alpha_{v}}(F_{v})=\coprod_{\alpha_{V}}G_{V,\alpha_{V}}(F_{V}).\]

\bigskip
Suppose that $G$ is a $K$-group over $F$, and  $G^{\ast}$ is a quasisplit inner twist of $G$. Then
$G^{\ast}$ is a connected quasisplit group over $F$, together with a inner class of inner twists $\psi_{\alpha}:G_{\alpha}\to G^{*}$
and a corresponding family of functions $u_{\alpha}:\Gamma\to G^{\ast}_{\ssc}$, for $\alpha\in \pi_{0}(G)$.
Then $G^{\ast}$ determines a quasisplit inner twist $G^{\ast}_{v}$ of each local $K$-group $G_{v}$. We shall refer to $G$ as an inner $K$-form of $G^{\ast}$.

\bigskip
If $\gamma_{V}$ is an element in the set $\Gamma_{G_{V}}(M_{V})$, let $\alpha_{V}\in\pi_{0}(M_{V})$ be the index such that $\gamma_{V}$ belongs to $\Gamma_{G_{V}}(M_{\alpha_{V}})$. We define the weighted orbital integral of $f$ at $\gamma_{V}$ by
\[ J_{M_{V}}(\gamma_{V},f_{V})=J_{M_{V}}(\gamma_{V},f_{\alpha_{V}}) \qquad  f_{V}\in \mathcal{H}(G_{V},\zeta_{V}),\]
where $J_{M_{V}}(\gamma_{V},f_{\alpha_{V}})$ is the weighted orbital integral on $G_{\alpha_{V}}(F_{V})$. Similarly, we set \[I_{M_{V}}(\gamma_{V},f_{V})=I_{M_{V}}(\gamma_{V},f_{\alpha_{V}})   \qquad  f_{V}\in \mathcal{H}(G_{V},\zeta_{V}),\]
where $I_{M_{V}}(\gamma_{V},f_{\alpha_{V}})$ is the invariant distribution on $G_{\alpha_{V}}(F_{V})$.

\bigskip
If $f_{V}=\oplus_{\alpha_{V}}f_{\alpha_{V}}$ in $ \mathcal{H}(G_{V},\zeta_{V})$, $f_{G_{V}}(\gamma_{V})$ denotes the invariant orbital integral $I_{G_{V}}(\gamma_{V},f_{\alpha_{V}})$,
\[f^{'}_{V}(\delta^{'}_{V})=f^{\widetilde{G}^{'}_{V}}_{V}(\delta^{'}_{V})=\sum_{\gamma_{V}\in\Gamma(G_{V})}\Delta_{G_{V}}(\delta^{'}_{V},\gamma_{V})f_{G_{V}}(\gamma_{V}),\qquad \delta^{'}_{V}\in \bigtriangleup_{G_{V}}(\widetilde{G}^{'}_{V}),\]
then $f^{'}_{V}=\oplus_{\alpha_{V}}f^{'}_{\alpha_{V}}$. The Langlands-Shelstad transfer Theorem, applied to each of the groups $G_{\alpha_{V}}$, asserts that $f_{V}^{'}$ belongs to the space $SI\mathcal{H}(\widetilde{G}^{'}_{V},\widetilde{\zeta}^{'}_{V})$ of stable orbital integrals of functions in $\mathcal{H}(\widetilde{G}^{'}_{V},\widetilde{\zeta}^{'}_{V})$, where $\widetilde{G}^{'}_{V}$ comes with a central data $(\widetilde{Z}^{'}_{V},\widetilde{\zeta}_{V}^{'})$.

\bigskip
  Similarly, we can define the objects on a $K$-group on the spectral side. We will stabilize the spectral side of the local trace formula for the $K$-group over $\mathbb{R}$ in section \ref{sec:tc}, section \ref{sec:se}, and section \ref{sec:cs}.

\bigskip
\section{Multiplicity of discrete series and invariant trace formula}  \label{sec:mi}

   Suppose that $F$ is a  number field, and $G$ is a connected reductive $K$-group over $F$. We can form the ad\'elic ring $\A_{F} = \prod_{v}' F_{v}$, and the group of ad\'elic points of $G$ is $G(\A_{F}) = \prod_{v}'G(F_{v})$.

Automorphic representations of $G$ over $F$ are irreducible constituents of the right regular representation of $R$, defined by
   \[R(x)\varphi(y)=\varphi(yx), \qquad  \varphi\in L^{2}(G(F) \backslash G(\A_{F})).\]
The fundamental problem in the automorphic representation theory is to decompose $R$. It is well known that $R$ can be decomposed into a discrete spectrum and a continuous spectrum,
  \[ R=R_{\disc}\oplus R_{\cont}.\]
  Langlands \cite{L3} has studied the continuous spectrum by Eisenstein series. So it remains to study the discrete part
 \begin{equation}
  R_{\disc}=\oplus_{\pi\in\Pi_{\disc}(G(\A))}m_{\disc}(\pi)\pi
\end{equation}
where $\Pi_{\disc}(G(\A))$ stands for the set of equivalence classes of irreducible representations of $G(\A)$ on $L^{2}_{\disc}(G(F) \backslash G(\A_{F}))$.

\bigskip
 Suppose $\pi$ is an irreducible automorphic representation of $G(\A_{F})$. Then we have a decomposition
 \[ \pi = \otimes_{v}^{\prime} \pi_{v},\]
 such that
 \begin{enumerate}
 \item $\pi_{v}$ is an irreducible admissible representation of $G(F_{v})$;
 \item $\pi_{v}$ is unramified for almost all $v$.
 \end{enumerate}

  We write $\pi=\pi_{\R}\otimes \pi_{\fin}$, where $\pi_{\R}$ and $\pi_{\fin}$ are irreducible representations of $G(\R)$ and $G(\A_{\fin})$ respectively. In this paper we study the multiplicity formula of $m_{\disc}(\pi_{\R})$ for connected $K$-group
 over $\Q$. The original problem comes from \cite[p,284]{A3}, where Arthur applied the invariant trace formula to compute the $L^{2}$-Lefschetz numbers of Hecke operators. He obtained the sum of multiplicity formula for $\sum_{\pi\in\Pi_{\disc}(\mu)}m_{\disc}(\pi_{\R},K_{0})$, under a weak regularity assumption on the representations in $\Pi_{\disc}(\mu)$, where $G$ is a connected reductive group, and $K_{0}$ is an open compact subgroup of the finite ad\'elic group $G(\mathbb{A}_{\fin})$. The packet $\Pi_{\disc}(\mu)$ consists of the set of discrete series representations with the same infinitesimal character $\mu$. We will give a formula for single multiplicity $m_{\disc}(\pi_{\R},K_{0})$ again under regularity assumption on infinitesimal character on $\pi_{\mathbb{R}}$. Since the double coset space
   \[G(\Q)\backslash G(\A)/G(\R)K_{0}\]
   is finite, we denote by ${x_{1}=1,x_{2}\cdot\cdot\cdot, x_{n}}$ the set of representatives in $G(\A_{\fin})$. The groups \[\Gamma_{i}=(G(\Q)\cdot x_{i}K_{0}x_{i}^{-1})\cap G(\R), 1\leq i\leq n\] are arithmetic subgroups of $G(\R)$, and $G(\Q)\backslash G(\A)/K_{0}$ is the disjoint union of space $\Gamma_{i}\backslash G(\R)$.
 The question of the multiplicity is quite natural from the point of view of spectral theory. More generally, one can consider Hecke operators $h$ on $L^{2}(G(F)\backslash{G(\A)})$, where $h$ is a $K_{0}$-bi-invariant function in $\mathcal{H}(\A_{\fin})$. Any such operator commutes with the action of $G(\R)$. Its restriction to the subspace $R_{\disc}(\pi_{\R})$ is denoted as $R_{\disc}(\pi_{\R},h)$.

\bigskip
 If $G$ is a connected reductive group over $\Q$, the multiplicities of discrete series have a homological interpretation. The global multiplicity $m_{\disc}(\pi)$ occurs in the well known isomorphism
 \[\mathnormal{H}^{q}_{(2)}(h,\mathcal{F}_{\mu})\cong\bigoplus_{\pi\in\Pi(G(\A),\zeta)}(m_{\disc}(\pi)\dimm\mathnormal{H}^{q}(\mathfrak{g}(\R),K_{\R};\pi_{\R}\otimes \mu))\pi_{\fin}(h).\]
    (see \cite[\S2]{A3}, the coefficient $m_{\disc}(\pi)$ stands for the multiplicity of $\pi$ in $R_{\disc}$), where $\Pi(G(\A),\zeta)$ is the set of equivalence classes of irreducible representations of $G(\mathbb{A})$, whose the central character coincides with a given quasi-character $\zeta$ on $Z$. $\mathfrak{g}$ is the Lie algebra of $G$; Both of $\mathfrak{g}(\R)$ and $K_{\R}$ act on the space of $K_{\R}$-finite vectors of the representation $\pi_{\R}\otimes \mu$ of $G(\R)$. The relative Lie algebra cohomology groups $\mathnormal{H}^{q}(\mathfrak{g}(\R),K_{\R};\pi_{\R}\otimes \mu)$ give the contribution of $\pi_{\R}$ to the cohomology. $V(\pi^{K_{0}}_{\fin})$ denotes the subspace of vectors in the underlying space of $\pi_{\fin}$ which are fixed by $K_{0}$, this is a finite dimensional subspace, which gives the contribution of $\pi_{\fin}$ to the cohomology. If the multiplicity of $\pi_{\R}$ occurs discretely in the representation of $G(\R)$ on $L^{2}(G(\Q)\backslash G(\A)/K_{0},\zeta)$, then
    \begin{equation} \label{eq:mr}
    m_{\disc}(\pi_{\R})=\sum_{\substack{
                \pi=\pi_{\R}\otimes \pi_{\fin}\\
                \pi\in\Pi(G(\A),\zeta)}}m_{\disc}(\pi)\dimm(V(\pi^{K_{0}}_{\fin})).
  \end{equation}

  \bigskip
  A representation of $K$-group $G$ is determined by the representations of connected components of $G$. If $\pi_{\alpha}\in\Pi(G_{\alpha},\zeta), f_{\alpha}\in C(G_{\alpha},\zeta), f\in C(G,\zeta)$, then we define \[ f_{G}(\pi_{\alpha})=f_{G_{\alpha}}(\pi_{\alpha}),\quad f=\oplus_{\alpha} f_{\alpha},  \qquad \alpha\in\pi_{0}(G).\]
So we can naturally extend the homological interpretation for the representation on the connected components of $K$-group to the $K$-group. Then we can extend \eqref{eq:mr} to $K$-group.

\bigskip
We write $\Pi_{2}(G(\R),\zeta)$, $\Pi_{\temp}(G(\R),\zeta)$ and $\Pi(G(\R),\zeta)$ for the set of equivalence classes of discrete series representations, the set of equivalence classes of tempered representations and the set of equivalence classes of all irreducible admissible representations of $G(\R)$, whose central character coincides with a given character $\zeta$ on $Z$. $C(G(\R),\zeta)$ is the Schwartz space of $G(\mathbb{R})$, whose central character is given by a character $\zeta^{-1}$ on $Z$.

 If $f$ is any function in $C(G(\R),\zeta)$ and $\pi_{\R}$ belongs to  $\Pi(G(\R),\zeta)$, we can set
 \[\pi_{\R}(f)=\int_{G(\R)/Z}f(x)\pi_{\R}(x)\,dx \]

 \bigskip
\begin{lemma}  \label{pseudocoefficient}
Let $\pi_{\R}\in\Pi_{2}(G(\R),\zeta)$, there is a function
$f_{\pi_{\R}}\in C(G(\R),\zeta)$, such that for any $\sigma\in \Pi_{\temp}(G(\R),\zeta)$,
\[ \tr(\sigma(f_{\pi_{\R}}))=
  \begin{cases}
   1 & \text{if }  \sigma=\pi_{\R}, \\
   0       & \text{ otherwise} .
  \end{cases}
\]
\end{lemma}

The above lemma is an immediate consequence of the trace Paley-Wiener theorem of Arthur \cite{A6}. Such $f_{\pi_{\R}}$ is called for a pseudo-coefficient of $\pi_{\R}$. We say that a function $f\in C(G(\R),\zeta)$ is cuspidal, if $\tr\pi(f)$ is viewed as a function on $\Pi_{\temp}(G(\R),\zeta)$, is supported on $\Pi_{2}(G(\R),\zeta)$. So $f_{\pi_{\R}}$ is cuspidal.

\bigskip
 Now, we fix a function $h\in\mathcal{H}(G(\A_{\fin}))$, and set \[(f_{\pi_{\R}}h)(x)=f_{\pi_{\R}}(x_{\R})h(x_{\fin}),\qquad x=x_{\R}x_{\fin}\in G(\A), x_{\R}\in G(\R), x_{\fin}\in G(\A_{\fin}). \]
 If $\pi=\pi_{\R}\otimes\pi_{\fin}$ is any representation in $\Pi(G(\A),\zeta)$, we have
\begin{align*} \tr\pi(f_{\pi_{\R}}h)&=\tr(\int_{G(\A)/Z}(f_{\pi_{\R}}h)(x)\pi(x)\,dx)\\
&=\tr\pi_{\R}(f_{\pi_{\R}})\tr\pi_{\fin}(h).
\end{align*}

Since $f_{\pi_{\R}}$ is cuspidal, which will cancel the contributions from Levi subgroups, so the invariant trace formula simplifies. The spectral side of the invariant trace formula is
\begin{align*}
 I(f_{\pi_{\R}}h) &= \sum_{\pi\in \Pi(G(\A),\zeta)}m_{\disc}(\pi)\tr\pi(f_{\pi_{\R}}h) \\
 &= \sum_{\pi\in\Pi(G(\A),\zeta)}m_{\disc}(\pi)\tr\pi_{\R}(f_{\pi_{\R}})\tr\pi_{\fin}(h)\\
  &= \sum_{\pi=\pi_{\R}\otimes\pi_{\fin}}m_{\disc}(\pi)\tr\pi_{\fin}(h).
\end{align*}

If we take $h$ for the characteristic function $I_{K_{0}}$ in $\mathcal{H}(\A_{\fin})$, then $\tr\pi_{\fin}(h)=\dimm(V(\pi^{K_{0}}_{\fin}))$. So
\begin{equation} \label{eq:mii}
I(f_{\pi_{\R}}I_{k_{0}}) =m_{\disc}(\pi_{\R},K_{0}),
\end{equation}
and
\begin{equation}\label{eq:tif}
I(f_{\pi_{\R}}h)=\tr(R_{\disc}(\pi_{\R},h)).
\end{equation}
We can expand $m_{\disc}(\pi_{\R},K_{0})$ and $\tr(R_{\disc}(\pi_{\mathbb{R}},h))$ by using the geometric side of invariant trace formula, but we cannot obtain an explicit formula directly. This is because $f_{\pi_{\R}}$ is not a stable function, so we can not cancel the contribution of the non-trivial unipotent elements in the invariant trace formula.

\bigskip
 If the test function $f$ is stable cuspidal. We can obtain an explicit formula through the geometric side of the invariant trace formula \cite{A3}, and the results are easily extended to $K$-group. In the rest of this section, we review the setup in \cite{A3}. Firstly, if $G$ is a connected $K$-group over $\R$, $f\in C(G(\R),\zeta)$ is stable cuspidal and $\gamma\in M(\R)$, where $M(\R)$ is a Levi subgroup of $G(\R)$, then following the section $4$ and section $5$ of \cite{A3}, when $\gamma$ is $G$-regular, one defines
 \[\Phi_{M}(\gamma,f)=|D^{M}(\gamma)|^{-1/2}I_{M}(\gamma,f).\]

 Then one has the formula:
 \begin{equation} \label {eq:idf}
 \Phi_{M}(\gamma,f)=(-1)^{\dimm(A_{M}/A_{G})}\upsilon{(M_{\gamma})^{-1}}\sum_{\uptau\in \Pi(G^{\ast}(\R),\zeta)}\Phi_{M}(\gamma,\uptau)\tr\widetilde{\uptau}(f),
 \end{equation}
 and $\Phi_{M}(\gamma,f)$ is equal to zero if $\gamma$ is not semisimple (see \cite[Theorem 5.1]{A3}). Here we denote by \[D^{M}(\gamma)=\de((1-Ad(\sigma))_{\mathfrak{m}/\mathfrak{m}_{\sigma}}),\] $\gamma$ is an element in $M(\R)$ with the Jordan decomposition $\gamma=\sigma u$, $G^{\ast}$ is isomorphic to $G$ over $\mathbb{C}$, and $G^{\ast}/A_{G}(\R)$ is compact. $\widetilde{\uptau}$ is the contragredient of $\uptau$. The other terms are explained as follow:
\[ \upsilon(G)=(-1)^{q(G)}\vol(G^{\ast}/A_{G}(\R)^{\circ})|\mathscr{D}(G,B)|^{-1},\]
\[\Phi_{M}(\gamma,\uptau)=(-1)^{q(G)}|D^{G}_{M}(\gamma)|^{1/2}\sum_{\pi\in{\Pi}_{\disc}(\uptau)}\Theta_{\pi}(\gamma),\]
where $\Pi_{\disc}(\uptau)$ is an $L$-packet of discrete series of $G(\R)$ parameterized by $\uptau$, \[q(G)=\frac{1}{2}\dimm(G(\R)/K_{\R}A_{G}(\R)),\]
$B$ is a maximal torus of $G(\R)$ that is contained in $K_{\R}$, and $K_{\R}$ is a maximal compact subgroup of $G(\R)$. The invariant distribution $\Phi_{M}(\gamma,\uptau)$ is given in terms of Harish-Chandra's formula for stable character of discrete series. Let us review the explicit formula for the averaged discrete series characters. We can naturally extend it to the $K$-group $G(\mathbb{R})$. Let $Z(B)$ be the centralizer of the connected component $G(\mathbb{R})^{\circ}$ in $K_{\mathbb{R}}.$ $B(\mathbb{R})$ (refer to \cite[Lemma 3.4]{Harish-Chandra1}) equals the product of its connected component $B(\mathbb{R})^{\circ}$ with $Z(B)$. We set $\rho_{B}$ as usual to be half of the sum of positive roots of $(G,B)$. Let $\Lambda(\zeta)$ denote the set of pairs
     \[(\zeta,\lambda),   \quad \zeta\in Z(B)^{\ast}, \lambda\in \mathfrak{b}(\mathbb{C})^{\ast},\]
such that $z\eexp H \to \zeta(z)e^{(\lambda-\rho_{B})(H)}, z\in Z(B), H\in \mathfrak{b}(\mathbb{R})$ is a well defined quasi-character on $B(\mathbb{R})$ whose restriction to $A_{G}(\mathbb{R})^{\circ}$ equals $\zeta$, and $\lambda$ is regular. $\Lambda(\zeta)$ equipped with an action of Weyl group $W(G,B)$. The discrete series are parameterized by the $W(G(\mathbb{R}),B(\mathbb{R}))$-orbits in $\Lambda(\zeta)$. We denote by $\phi$ a discrete Langlands parameter, and then we can find that an $L$-packet $\Pi_{\phi}$ corresponds to the partition of a given $W(G,B)$-orbit into $W(G(\mathbb{R}),B(\mathbb{R}))$-orbits \cite{S1}. We set $\mathscr{D}(G,B)=W(G,B)/W(G(\mathbb{R}),B(\mathbb{R}))$, then
        \[ |\mathscr{D}(T)|=|\mathscr{D}(G,B)|,\]
where $T$ is a maximal torus which is $\mathbb{R}$-anisotropic modulo $A(\mathbb{R})^{\circ}$, and $T$ is conjugate with $B$.

\bigskip
  Let $\eta: G^{\ast} \rightarrow G$ be an isomorphism over $\C$ such that the automorphism $\eta^{\sigma}\eta^{-1}$ is inner for $\sigma\in \Gal(\C/R)$. We use $\eta$ to identity $A_{G}$ with the $\R$-split component of the center of $G^{\ast}$, and we assume that $G^{\ast}(\R)/A_{G}(\R)$ is compact. Then the representations in $\Pi(G^{\ast},\zeta)$ are all finite dimensional, According to Langlands classification \cite{L2}, the set $\Pi_{\disc}(G(\R),\zeta)$ is a disjoint union of finite subsets $\Pi_{\disc}(\uptau)$, which are parametrized by the irreducible representation $\uptau$ in $\Pi(G^{\ast}(\R),\zeta)$. If $\uptau$ and $\phi$ are parametrized by the same infinitesimal character, then the $L$-packet $\Pi_{\phi}$ equals the finite set $\Pi_{\disc}(\uptau)$. We can set $\Phi_{M}(\gamma,\phi)=\Phi_{M}(\gamma,\uptau)$.

\bigskip
 For given $\uptau\in \Pi(G^{\ast}(\R),\zeta)$, let $(\zeta,\lambda)\in \Lambda(\zeta)$ be the point in the corresponding orbit, such that $\lambda$ is positive on all the positive co-roots of $(G,B)$. Then if
 \[     \gamma=z\eexp H , z\in Z(B), H\in \mathfrak{b}(\mathbb{R})\]
 is a regular point in $B(\mathbb{R})$, we have
 \[\Phi_{G}(\gamma,\uptau)=\tr\uptau(\gamma)=\Delta^{G}_{B}(H)^{-1}\zeta(z)\sum_{s\in W(G,B)}\varepsilon(s)e^{(s\lambda)(H)}.\]
Here $\Delta^{G}_{B}(H)=\Pi_{\alpha>0}(e^{\frac{1}{2}\alpha(H)}-e^{-\frac{1}{2}\alpha(H)}).$

\bigskip
For the general averaged discrete series character $\Phi_{M}(\gamma,\uptau)$, let $T$ be a maximal torus in $M$, which is $\mathbb{R}$-anisotropic modulo $A_{M}(\mathbb{R})^{\circ}$. We take $R$ to be the set of real roots of $(G,T)$. The existence of torus $B$ means that $W(R)$ contains an element that acts as $-1$. We can take $T$ from its $M(\mathbb{R})$-conjugates so that $T=(T\cap B)A_{M}$. Then there is an element $y\in G(\mathbb{C})$, such that $Ad(y)(\mathfrak{b}(\mathbb{C}))=\mathfrak{t}(\mathbb{C})$, where $\mathfrak{t}$ is the Lie algebra of $T$.

Suppose that $\uptau\in \Pi(G(\mathbb{R}),\xi)$. $(\xi,\lambda)\in \Lambda(\xi)$ is a point in the corresponding $W(G,B)$-orbit such that $y\lambda$ is positive on all positive co-roots of $(G,T)$. Then $\Phi_{M}(\gamma,\uptau)$ vanishes for any regular point $\gamma\in T(\mathbb{R})$ unless $\gamma$ is of the form
\[\gamma=z\eexp(H), \quad z\in Z(B), H\in \mathfrak{t}(\mathbb{R}),\]
in which case
\begin{equation}
\Phi_{M}(\gamma,\uptau)=\Delta^{M}_{T}(H)^{-1}\varepsilon_{R}(H)\xi(z)\sum_{s\in W(G,B)}\varepsilon(s)\bar{C}(Q^{+}_{ys\lambda},R^{+}_{H})e^{(ys\lambda)(H)}.
\end{equation}
Here $\varepsilon_{R}(H)=(-1)^{|R^{+}_{H}\cap (-R^{+})|}$. $H$ is a regular point in $\mathfrak{t}(\mathbb{R})$, which is the Lie algebra of $T(\mathbb{R})$, and $R^{+}_{H}$ for the set of roots which are positive on $H$. $\varepsilon(s)$ is a sign function on $W(G,B)$, and $\bar{C}(Q^{+},R^{+})$ is an integer valued function, which is defined for root systems $R$ whose Weyl group $W(R)$ contains $-1$. The function $\bar{C}(Q^{+},R^{+})$ is uniquely determined by the following four properties.
\begin{enumerate}
\item $\bar{C}(sQ^{+},sR^{+})=\bar{C}(Q^{+},R^{+}), s\in W(R)$.
\item The number $\bar{C}(Q^{+},R^{+})$ vanishes unless $\nu(X)$ negative for every $X\in \mathfrak{a}_{R^{+}}$, and $\nu \in \mathfrak{a}_{Q^{+}}$.
\item  $\bar{C}(Q^{+},R^{+})+\bar{C}(s_{\alpha}Q^{+},R^{+})=2\bar{C}(Q^{+}\cap Q_{\alpha},R^{+}\cap R_{\alpha})$, for any reflection $s_{\alpha}\in W(R)$ corresponding to a root $\alpha\in R$.
\item If $R$ is the empty root system, then $\bar{C}(Q^{+},R^{+})=1$.
\end{enumerate}
 Here $R^{+}$ is a system of positive roots for $R$, and $Q^{+}$ is a positive system for the set $Q=R^{\vee}$ of co-roots, $Q^{+}_{\lambda}$ is the set of co-roots $\alpha^{\vee}$ of $Q=R^{\vee}$ for which $\lambda(\alpha^{\vee})$ are all positive.

\bigskip
 Arthur \cite{A3} extended $\Phi(\gamma,\uptau)$ to a continuous, $W(M,T)$-invariant function on $T(\R)$, and then extended to a function on $M(\R)$, which is a connected reductive group, and $\Phi(.,\uptau)$ is a $M(\mathbb{R})$-invariant function on $M(\mathbb{R})$ which is supported on the $M(\R)$-elliptic conjugacy classes. If $M\in\mathcal{L}$ is not cuspidal, then we set $\Phi_{M}(\gamma,\uptau)$ to be identically zero.

  Finally we come to the geometric side of the invariant trace formula. This will be reviewed in more details in section $4$ below. If the representative of the conjugacy class $\gamma\in \Gamma(M,S)$ is semisimple, it is independent of $S$. Moreover, for any semisimple element $\gamma\in M(\Q)$,  \[\mathnormal{a}^{M}(S,\gamma)=|\iota^{M}(\gamma)|^{-1}\vol(M_{\gamma}(\Q)\backslash{M_{\gamma}(\A)^{1}}) . \]

\bigskip
If $\gamma$ is not $\Q$-elliptic in $M$, then $\mathnormal{a}^{M}(S,\gamma)$ vanishes. Here
$|\iota^{M}(\gamma)|=|M_{\gamma}(\Q)\backslash{M(\Q,\gamma)}|$, $M_{\gamma}(\A)^{1}=A_{M}(\R)^{0}\backslash{M_{\gamma}(\A)}$.
 So if $f$ is stable cuspidal, we have an explicit geometric expansion of general global trace formula,
\begin{equation} \label{eq:inf}
 I(fh)=\sum_{M\in\mathcal{L}}|W^{M}_{0}||W^{G}_{0}|^{-1}\sum_{\gamma\in \Gamma(M,S)}a^{M}(S,\gamma)I_{M}(\gamma,fh),
\end{equation}
and $I_{M}(\gamma,fh)=I^{G}_{M}(\gamma,f)I^{M}_{M}(\gamma,h)=|D^{M}(\gamma)|\Phi_{M}(\gamma,f)h_{M}(\gamma)$.

\bigskip
\section{Global stable trace formula }  \label{sec:gs}

 Suppose now $G$ is a $K$-group over a number field $F$. The invariant trace formula can be extended to $K$-group, which is stabilized by Arthur in \cite{A9}, \cite{A10}, \cite{A11}. We need to recall the basic information about the stable trace formula before applying it to obtain the multiplicity formula. Let $S$ be a finite set of valuations of $F$ that contains the set of archimedean places and the set of places at which $G$ ramifies. The general invariant trace formula is the identity obtained from two different expansions of a certain linear form $I(f)$ in \cite{A2} for $f\in\mathcal{H}(G,S)$, the Hecke algebra of $G(F_{S})$. The geometric expansion
\begin{equation}
I(f)=\sum_{M\in\mathcal{L}}|W^{M}_{0}||W^{G}_{0}|^{-1}\sum_{\gamma\in\Gamma(M,S)}a^{M}(S,\gamma)I_{M}(\gamma,f)
 \end{equation}
is a linear combination of distributions parametrized by conjugacy classes $\gamma$
in Levi subgroups $M(F_{S})$. The spectral expansion
\begin{equation}
I(f)=\sum_{M\in\mathcal{L}}|W^{M}_{0}||W^{G}_{0}|^{-1}\sum_{t\geq 0}\int_{\Pi_{t}(M,S)}a^{M}(S,\pi)I_{M}(\pi,f)\,d\pi
\end{equation}
is a linear combination of distributions parameterized by representations $\pi$ of Levi subgroup $M(F_{S})$.

\bigskip
Arthur studied the more general coefficients $a^{M}(\gamma)$ and $a^{M}(\pi)$, which are the global objects. The terms $I_{M}(\gamma,f)$ and $I_{M}(\pi,f)$ are the local objects. We can apply splitting formula and descent formula to reduce the distribution to Levi subgroup of $G$. The multiplicity formula results are described by the geometric side of the global trace formula.

\bigskip
We need to stabilize the coefficient $a^{M}(\gamma)$ and the invariant distribution $I_{M}(\gamma,f)$. Those are the corresponding global theorem and local theorem \cite{A11}. Firstly, we recall the relation between the coefficient $a^{M}(\gamma)$ and $a^{M}(S,\gamma)$. We will freely use the notation in \cite{A9}. It suffices to consider $M=G$. The coefficient $a^{G}(S,\gamma)$ is defined on $\Gamma(G,S)$, where
$\Gamma(G,S)$ is the set of the $(G,S)$-equivalence classes in $G(F)$. The two elements $\gamma$ and $\gamma_{1}$ in $G(F)$, with standard Jordan decompositions $\gamma=c\alpha$ and $\gamma_{1}=c_{1}\alpha_{1}$, are defined to be $(G,S)$-equivalent if there is an element $\delta\in G(F)$ such that $\delta^{-1}c_{1}\delta=c$, and such that $\delta^{-1}\alpha_{1}\delta$ is conjugate to $\alpha$ in $G_{c}(F_{S})$. It is the usual conjugacy class if $\gamma$ is semisimple.

\bigskip
For a general element $\gamma=c\alpha$, $c$ being the semisimple part, $\alpha$ being the unipotent part, the coefficient is defined by a descent formula, \[a^{G}(S,\gamma)=i^{G}(S,c)|\stab(c,\alpha)|^{-1}a^{G_{c}}(S,\alpha),\]
where $\stab(c,\alpha)$ stands for the stabilizer of $\alpha$ in the finite group $(G_{c,+}(F)/G_{c}(F))$, which acts on the set of unipotent conjugacy classes in $G_{c}(F_{S})$. $i^{G}(S,c)$ is equal to 1, if $c$ is $F$-elliptic in $G$, and the $G(\A^{S})$ conjugacy class of $c$ meets $K^{S}$; otherwise equal to 0. The descent formula reduces the study of $a^{G}(S,\gamma)$ to the case of unipotent elements.

\bigskip
 We set
 \[ a^{G}_{\elll}(\gamma)=\sum_{\{\gamma\}}|Z(F,\gamma)|^{-1}a^{G}(S,\gamma)(\gamma_{S}/\gamma)^{-1},\]
where $\{\gamma\}$ is summed over those $Z_{S,O}=Z(F)\cap Z_{S}Z(O)^{S}$ orbit in $(G(F))_{G,S}$ that map to $\gamma_{S}$, and such that the $G(\A^{S})$ conjugacy class of $\gamma$ in $G(\A^{S})$ meets $K^{S}$, $Z(F,\gamma)=\{z\in Z(F):z\gamma=\gamma\}=\{z\in
Z_{S,O}:z\gamma=\gamma\}$, and $\gamma_{S}/\gamma$ is the ratio of the invariant measure on $\gamma_{S}$ and the signed measure on $\gamma_{S}$ that comes with $\gamma$. The coefficient $a^{G}_{\elll}(\gamma)$ is supported on the set of admissible elements in the discrete subset $\Gamma_{\elll}(G,S,\zeta)$ (see \cite[(2.6)]{A9}) of $\Gamma(G^{Z}_{S},\zeta_{V})$, where admissible elements is defined in \cite[\S 1]{A9}. We denote by
 \[G^{Z}=\{x\in G: H_{G}(x)\in \image(a_{Z}\mapsto a_{G})\}.\]

If $M$ is a Levi subgroup of $G$, and $\mu$ belongs to $\Gamma(M^{Z}_{S},\zeta_{S})$, the induced distribution $\mu^{G}$ is a finite linear combination of elements in $\Gamma(G^{Z}_{S},\zeta_{S})$. We write $\Gamma(G,S,\zeta)$ for the set of elements so obtained, as $M$ ranges over $\mathcal{L}$ and $\mu$ runs over the element $\Gamma_{\elll}(M,S,\zeta)$, these objects are compatible with the spectral side.

\bigskip
If $\gamma$ belongs to $\Gamma(G^{Z}_{V},\zeta_{V})$, $V_{\ram}\subset V\subset S$, we denote \cite[(2.8)]{A9}
 \[ a^{G}(\gamma)=\sum_{M\in\mathcal{L}}|W^{M}_{0}||W^{G}_{0}|^{-1}\sum_{k\in\mathcal{K}^{V}_{\elll}(M/Z,S)}a^{M}_{\elll}(\gamma_{M}\times k)r^{G}_{M}(k),\]
where $r^{G}_{M}(k)=\mathcal{J}_{M}(r^{V}_{S}(k),u^{V}_{S}), k\in\mathcal{K}((M/Z)^{V}_{S})$ is the unramified weighted orbital integrals, $\mathcal{K}^{V}_{\elll}(G/Z,S)$ for the set of $k$ in $\mathcal{K}((G/Z)^{V}_{S})$ such that $\gamma\times k$ belongs to $\Gamma_{\elll}(G,S,\zeta)$ for some $\gamma$.

If $f\in\mathcal{H}(G,V,\zeta)$, then the linear form $I(f)$ has a geometric expansion \cite[Proposition 2.2]{A9}
\begin{equation}I(f)=\sum_{M\in\mathcal{L}}|W^{M}_{0}||W^{G}_{0}|^{-1}\sum_{\gamma\in\Gamma(M,V,\zeta)}a^{M}(\gamma)I_{M}(\gamma,f).
  \end{equation}
The general trace formula was stabilized by Arthur, who inductively defined a general stable distribution on the quasisplit group $G^{\prime}$ which is independent of the $G$, and then built up the endoscopy trace formula.

\bigskip
If $G$ is quasisplit, Arthur \cite{A11} proved that
 \[S^{G}(f)=I(f)-\sum_{G^{'}\in\mathcal{\mathcal{E}}^{0}_{\elll}(G)}\iota(G,G^{'})\widehat{S}^{\widetilde{G}^{'}}(f^{'})\]
is stable.
For $G$ general, he \cite{A11} proved
 \[ I(f)=I^{\mathcal{E}}(f)=\sum_{G^{'}\in\mathcal{\mathcal{E}}_{\elll}(G)}\iota(G,G^{'})\widehat{S}^{G^{'}}(f^{'}).\]
The endoscopic trace formula $I^{\mathcal{E}}(f)$ and the stable trace formula $S^{G}(f)$ both have a geometric expansion
\begin{equation} \label{eq:edte}
  I^{\mathcal{E}}(f)=\sum_{M\in\mathcal{L}}|W^{M}_{0}||W^{G}_{0}|^{-1}\sum_{\gamma\in\Gamma^{\mathcal{E}}(M,V,\zeta)}a^{M,\mathcal{E}}(\gamma)I^{\mathcal{E}}_{M}(\gamma,f),
\end{equation}
and
\begin{equation} \label{eq:ste}
  S^{G}(f)=\sum_{M\in\mathcal{L}}|W^{M}_{0}||W^{G}_{0}|^{-1}\sum_{\delta\in\Delta(M,V,\zeta)}b^{M}(\delta)S^{G}_{M}(\delta,f).
  \end{equation}
Here the relation between coefficients is given by  \[a^{G}(\gamma)=a^{G,\mathcal{E}}(\gamma)=\sum_{G^{'}}\sum_{\delta^{'}}\iota(G,G^{'})b^{\widetilde{G^{'}}}(\delta^{'})\Delta_{G}(\delta^{'},\gamma)+\varepsilon(G)\sum_{\delta}b^{G}(\delta)\Delta_{G}(\delta,\gamma),
\]
with $\gamma\in\Gamma(G^{Z}_{V},\zeta_{V})$, $G^{'}, \delta^{'}$ and $\delta$ summed over $\mathcal{E}^{0}_{\elll}(G)=\mathcal{E}_{\elll}\backslash \{G^{\ast}\}$, $\Delta((\widetilde{G^{'}}_{V})^{Z^{'}},\widetilde{\zeta}^{'}_{V})$ and $\Delta^{\mathcal{E}}(G^{Z}_{V},\zeta_{V})$ respectively, where $G^{\ast}$ is the quasisplit inner form of $G$. The stable coefficient $b^{G}(\delta)$ is stable, meaning that $b^{G}(\delta)$ is supported on $\Delta(G,V,\zeta)$, and the coefficients $a^{G}(\gamma)$ and $b^{G}(\delta)$ are independent of $S$.

\bigskip
To stabilize the general trace formula, it essentially amounts to comparing the two expansions \eqref{eq:edte} and \eqref{eq:ste}. Assuming the Fundamental Lemma and the weighted Fundamental Lemma, Arthur \cite{A11} has proved this, and in 2008 Ngo \cite{N} proved the Fundamental Lemma. The weighted version was proved by Chaudouard-Laumon. So now the stable trace formula is available.

\bigskip
  We have the formula for the stable coefficient
   \[ b^{G}(\delta)= b^{G}_{\elll}(\delta,S)+\sum_{M\in \mathcal{L}^{0}}|W^{M}_{0}||W^{G}_{0}|^{-1}\sum_{k\in\mathcal{K}^{V}_{\elll}(M/Z,S)}b^{M}_{\elll}(\delta_{M}\times k)S^{G}_{M}(k),\]
where the stable term $S^{G}_{M}(k)$ comes from the unramified weighted orbital integrals $r^{G}_{M}(k)$ (the weighted Fundamental Lemma was applied to this part).
  Here $b^{G}_{\elll}(\delta)$ also satisfies a descent formula \cite{A10}, if $\delta_{S}$ is an admissible element in $\Delta_{\elll}(G,S,\zeta)$ with Jordan decomposition $\delta_{S}=d_{S}\beta_{S}$, then \[b^{G}_{\elll}(\delta_{S})=\sum_{d}\sum_{\beta}j^{G^{\ast}}(S,d)b^{G^{\ast}_{d}}_{\elll}(\beta),\]
where $d$ is summed over the set of elements in $\Delta_{ss}(G^{\ast})$ whose image in $\Delta_{\sss}(G^{\ast}_{S})$ equals $d_{S}$, where $G^{\ast}$ is quasisplit inner form of $G$, and $\beta$ is summed over the orbit of $(\overline{G^{\ast}}/\overline{G^{\ast}}_{d,S})(F)$ in $\Delta_{\unip}(G^{\ast}_{d,S},\zeta)$. Moreover, $b^{G}_{\elll}(\delta)$ is equal to zero if $\delta$ in the complement of $\Delta_{\elll}(G,S,\zeta)$ in the set of admissible elements in $\Delta^{\mathcal{E}}_{\elll}(G,S,\zeta)$ (see \cite[Theorem 1.1]{A10}), and $j^{G^{\prime}}(S,d^{\prime})=i^{G^{\prime}}(S,d^{\prime})\tau(G^{\prime})\tau(G^{\prime}_{d^{\prime}})^{-1}$ (see \cite[(1.7)]{A10}).

If $\delta^{\prime}$ is semisimple, elliptic in $\Delta_{\elll}(G^{\prime},V,\zeta)$, then by \cite[Theorem 8.3.1]{K2}, or \cite[page 105]{A11},
\begin{align*}
b^{G^{\prime}}(\delta^{\prime})&=b^{G^{\prime}}_{\elll}(\delta^{\prime})=j^{G^{\prime}}(S,\delta^{\prime})b^{G^{\prime}_{\delta^{\prime}}}_{\elll}(1) \\
 &= \tau(G^{\prime})\tau(T)^{-1}\tau(T)=\tau(G^{\prime}),
\end{align*}
where $T=G^{\prime}_{\delta^{\prime}}$ and $\tau(G^{\prime})=|\pi_{0}(Z(\widehat{G'})^{\Gamma})||\Ker^{1}(F,Z(\widehat{G'}))|^{-1}$ is the Tamagawa number of $G^{\prime}$ \cite[(5.1.1)]{K1}, or \cite{K3}.

\bigskip
We obtain a stable form of the invariant trace formula
\[I(f)=\sum_{G^{\prime}\in\mathcal{E}_{\elll}(G)}\iota(G,G^{\prime})\sum_{M^{\prime}\in\mathcal{L}^{G^{\prime}}}|W^{M^{\prime}}_{0}||W^{G{\prime}}_{0}|^{-1}\sum_{\delta\in \Delta(M^{\prime},V,\zeta)}b^{M^{\prime}}(\delta)S^{G^{\prime}}_{M^{\prime}}(\delta,f).\]
 Where $\Delta(M',V,\zeta)$ is the basis of $S\mathcal{D}(M^{\prime Z}_{V},\zeta)$, which is the subspace of stable distributions in $\mathcal{D}(M^{\prime Z}_{V},\zeta)$.  Here $\mathcal{D}(M^{\prime Z}_{V},\zeta)$ is the vector space of distributions $D$ on $M^{\prime Z}_{V}$ that satisfy the following three conditions:
 \begin{enumerate}
   \item $D$ is invariant under the conjugation by $M^{\prime Z}_{V}$,
   \item $D$ is $\zeta$-equivariant under translation by $Z_{V}$,
   \item $D$ is supported on the preimage in $M^{\prime Z}_{V}$ of a finite union of conjugacy classes in $\overline{M^{\prime Z}_{V}}=M^{\prime Z}_{V}/Z_{V}$.
 \end{enumerate}
   The Langlands's global coefficients
 \[\iota(G,G^{\prime})=\iota(G_{\alpha},G^{\prime})=\tau(G)\tau(G^{\prime})^{-1}|\Out_{G}(G^{\prime})|^{-1},\]
 where $\alpha\in \pi_{0}(G), G^{\prime}\in\mathcal{E}_{\elll}(G), \Out_{G}(G^{\prime})=\Aut_{G}(G^{\prime})/\widehat{G^{\prime}}$.

\bigskip
If $f=f_{\infty}h$, $f_{\infty}$ is cuspidal in $C(G(\R),\zeta)$ and $h\in\mathcal{H}(G(\A_{\fin})),$ then the term
\[S^{G}_{M}(\delta,f)= S^{G}_{M}(\delta,f_{\infty})(f^{\infty}_{V})^{M}(\delta),   \qquad \delta\in\Delta(M,V,\zeta).\]

This is a consequence of the stable splitting formula \cite[Theorem 6.1]{A8}. Indeed we set $V=V_{1}\coprod V_{2}$, where $V_{1}$ is the set of archimedean places, $f=f_{V_{1}}f_{V_{2}}$, then \[ S^{G}_{M}(\delta,f)=\sum_{L_{1},L_{2}\in\mathcal{L}(M)}e^{G}_{M}(L_{1},L_{2})\widehat{S}^{L_{1}}_{M}(\delta,f_{V_{1},L_{1}})
\widehat{S}^{L_{2}}_{M}(\delta,f_{V_{2},L_{2}}),\]
where $e^{G}_{M}(L_{1},L_{2})$ was defined \cite[Theorem 6.1]{A8}. Since $f_{V_{1}}$ is cuspidal, thus $(f_{V_{1}})_{L_{1}}=0$  for $L_{1}\neq G$. On the other hand, $e^{G}_{M}(G,L_{2})\neq 0$ only when $L_{2}=M$ in which case it is equal to 1. Then
\[S^{G}_{M}(\delta,f)=\widehat{S}^{G}_{M}(\delta,f_{V_{1},G})\widehat{S}^{M}_{M}(\delta,f_{V_{2},M}).\]
However $V$ is finite, so we continue this process to obtain \[S^{G}_{M}(\delta,f_{\infty}h)=S^{G}_{M}(\delta,f_{\infty})(h^{\infty}_{V})^{M}(\delta).
\]

\bigskip
We now get the following proposition by combining \eqref{eq:tif}.
\begin{proposition}\label{prop:cmf}
For any $h\in\mathcal{H}(G(\A_{\fin}))$, and $f_{\pi_{\R}}\in C_{\cusp}(G(\R),\zeta)$ is pseudo-coefficient of $\pi_{\R}\in \Pi_{2}(G(\R),\zeta)$. We have
\begin{equation}
\begin{split}
\tr(R_{\disc}(\pi_{\R},h))&=I(f_{\pi_{\R}}h)\\
&=\sum_{G^{'}\in\mathcal{E}_{\elll}(G)}\iota(G,G^{'})\sum_{M^{'}\in\mathcal{L}^{G^{'}}}|W^{M^{'}}_{0}||W^{G^{'}}_{0}|^{-1}\sum_{\delta\in\Delta(M^{'},V,\zeta)}b^{M^{'}}(\delta)S^{G^{'}}_{M^{'}}(\delta,f_{\pi_{\R}})(h_{M})^{M^{'}}(\delta).
\end{split}
\end{equation}
\end{proposition}

\bigskip
We now have a stable trace formula. In order to make this more explicit, we need to examine the individual terms more closely. The stable orbital integral $(h_{M})^{M^{'}}(\delta)$ is treated by the transfer theorem. The term $S^{G}_{M}(\delta,f_{\pi_{\R}})$ is a stable distribution attached to a invariant distribution, which is more complicated. We need to give an explicit formula for $S^{G}_{M}(\delta,f_{\pi_{\R}})$. We know that the local component of stable global trace formula is more complicated than the local component of stable local trace formula. But if the test function is stable cuspidal, they are the same, because the component of invariant local trace formula and the local component of invariant trace formula are the same, when $G$ is a connected reductive group \cite{A3}. Moreover we have the following proposition.

\bigskip
\begin{proposition} \label{prop:compa}
If $G(\mathbb{R})$ is a $K$-group, $G^{\prime}$ is the endoscopy group of $G$, $f$ is a cuspidal function in $\mathcal{H}_{\ac}(G(\mathbb{R}),\zeta)$, $\delta\in M^{\prime}(\mathbb{R})$ is any element with Jordan decomposition $\delta=cu$, and $u$ is a non-trivial unipotent element, then $S^{G^{\prime}}_{M^{\prime}}(\delta,f)=0$.
\begin{proof}
   $G$ is a $K$-group over $\R$, we can write $G=\amalg_{\alpha\in\pi_{0}(G)}G_{\alpha}$, where $G_{\alpha}$ is a connected reductive group.
   If $\delta\in M^{\prime}_{\alpha}$ and  $\delta=cu$, $u$ is not trivial, then we have
   \[S^{G^{\prime}}_{M^{\prime}}(\delta,f)=S^{G^{\prime}_{\alpha}}_{M^{\prime}_{\alpha}}(\delta,f_{\alpha})=\widehat{S}^{G^{\prime}_{\alpha}}_{M^{\prime}_{\alpha}}(\delta,(f)^{G^{\prime}_{\alpha}})=0.\]
   The third equality is from Arthur's result \cite[Theorem 5.1]{A3}. Indeed we know that when $G$ is a connected reductive group, and when the test function is stable cuspidal, then the corresponding invariant distribution vanishes on the non-semisimple element. However $f^{G^{\prime}_{\alpha}}(\delta)$ is stable cuspidal function of $G^{\prime}_{\alpha}$ by Shelstad's transfer theorem, so we obtain $S^{G^{\prime}}_{M^{\prime}}(\delta,f)=0$.
\end{proof}
\end{proposition}

\bigskip
From the above proposition, we see that the component of stable local trace formula and the component of stable global formula are compatible. So we have reduced the study of the stable trace formula to the stable local trace formula in the archimedean case. There are two ways to study the stable distribution $S^{G}_{M}(\delta,f_{\infty})$. One way is to explicitly stabilize the invariant trace formula in this special case. The other is to stabilize the local trace formula, and then we compare the stable Weyl integral formula with the stable local trace formula. In this paper we follow the second way.

\bigskip
\section{ The transfer factors of spectral side and characters  } \label{sec:tc}

  Suppose that $G$ is a reductive $K$-group over $\R$, $Z$ stands for a central induced torus in $G$ over $\R$, $\zeta$ is a character on $Z(\R)$. Before we discuss the stabilization, we need to recall the invariant local trace formula\cite{A4} and virtual characters in the Archimedean case \cite{A5}. We set $V=\{\infty_{1},\infty_{2}\}$ for two Archimedean places. Then $G_{V}=G(\R)\times G(\R)$ and $\zeta_{V}=\zeta\times\zeta^{-1}$, while $f=f_{1}\times \bar{f}_{2}$, where $f_{1}, f_{2}\in C(G(\R),\zeta)$. Then
  $f$ is a function in the Schwartz space $C(G_{V},\zeta_{V})$.

   The geometric side of the local trace formula is the linear form
   \begin{equation}
   I(f)=\sum_{M\in \mathcal{L}}|W^{M}_{0}||W^{G}_{0}|^{-1}(-1)^{\dimm(A_{M}/A_{G})}\int_{\Gamma _{\elll}(M,V,\zeta)} I_{M}(\gamma,f)\,{d}\gamma,
   \end{equation}
   where $\Gamma_{\elll}(M,V,\zeta)=\{(\gamma,\gamma): \gamma\in \Gamma_{\elll}(M,\zeta)\}$. (The set $\Gamma_{\elll}(M,V,\zeta)$ is in bijection with the family  $\Gamma _{\elll}(\bar{M})$ of elliptic conjugacy classes in $\bar{M}(\R)= M(\R)/Z(\R).)$

\bigskip
The spectral side is the linear form
   \begin{equation} \label{eq:lsit}
   I(f)=\sum_{M\in \mathcal{L}}|W^{M}_{0}||W^{G}_{0}|^{-1}(-1)^{\dimm(A_{M}/A_
  {G})}\int_{T_{\disc}(M,V,\zeta)} i^{M}(\tau)I_{M}(\tau,f_{1}\times \bar{f_{2}})\,{d}\tau,
   \end{equation}
where $I_{M}(\tau,f_{1}\times \bar{f_{2}})=r_{M}(\tau,P)\theta(\tau,f_{1,P})\overline{\theta(\tau,f_{2,P})}$ as in \cite{A5}, $T_{\disc}(M,V,\zeta)$ stands for the diagonal image $\{(\tau,\tau^{\vee}): \tau \in T_{\disc}(M,\zeta)\}$ in $T_{\temp}(M_{V},\zeta_{V})$, defined as in \cite[\S 3]{A5}.

\bigskip
   We denote the leading term ( i.e. $M=G$ in \eqref{eq:lsit}) by $I_{\disc}(f)$. Then
   \begin{equation}\label{eq:infm}
   I_{\disc}(f)=\int_{T_{\disc}(G,V,\zeta)} i^{G}(\tau)f_{G}(\tau)\,{d}\tau ,
   \end{equation}
where $f_{G}(\tau)=(f_{1})_{G}(\tau)(\bar {f}_{2})_{G}(\tau^{\vee})= f_{1,G}(\tau)\overline {f_{2,G}(\tau)}$, and \[i^{G}(\tau)=|W^{0}_{\pi}|^{-1}|R_{\pi,r}|^{-1}\sum_{w\in W_{\pi}(r)_{\reg}}\varepsilon_{\pi}(w)|\de(1-w)_{\mathfrak{a}^{G}_{M}}|^{-1},\]
where $\tau=(M,\pi,r).$ \[W_{\pi}=\{w\in W(\mathfrak{a}_{M}):w\pi\cong \pi\}.\] $W^{0}_{\pi}$ is the subgroup of elements $w\in W_{\pi}$ such that the operator $R(w,\pi)$(see \cite[\S 2]{A5})
 is a scalar. $R_{\pi}=W_{\pi}/W^{0}_{\pi}$, $R_{\pi,r}$ is the centralizer of $r$ in the group $R_{\pi}$. $W_{\pi}(r)_{\reg}$ is the intersection of the $W^{0}_{\pi}$-coset $W_{\pi}(r)=W^{0}_{\pi}r$ in $W_{\pi}$ with the set \[W_{\pi,\reg}=\{w\in W_{\pi}: \mathfrak{a}^{w}_{M}=\mathfrak{a}_{G}\}\] of regular elements. $\varepsilon_{\pi}(w)$ stands for the sign of projection of $w$ onto the Weyl group $W^{0}_{\pi},$ taken relative to the decomposition $W_{\pi}=W^{0}_{\pi}\rtimes R_{\pi}$.

\bigskip
If $f=f_{1}\times \bar{f}_{2}$ and $f_{1}\in C_{\cusp}(G(\R),\zeta),f_{2}\in C(G(\R),\zeta)$, then \[I(f)=I_{\disc}(f).\]
Here \[I(f)=\sum_{M\in \mathcal{L}}|W^{M}_{0}||W^{G}_{0}|^{-1}(-1)^{\dimm(A_{M}/A_{G})}\int_{\Gamma _{G-\reg,\elll}(M,V,\zeta)} I_{M}(\gamma,f)\,{d}\gamma,\]
where $\Gamma_{G-\reg,\elll}(M,V,\zeta)$ is the subset of strongly $G$ regular, elliptic elements in the basis $\Gamma(M,\zeta),$
and  $i^{G}(\tau)=|d(\tau)|^{-1}|R_{\pi,r}|^{-1},$        $d(\tau)=d(r)=\de(1-r)_{\mathfrak{a}_{M}/\mathfrak{a}_{G}}$.

\bigskip
We write $\widehat{G}$ for the complex dual of $G$, and ${}^{L}G$ for the $L$-group $\widehat{G}\rtimes W_{\R}$, which acts through $W_{\R}\to\Gamma$, Galois group $\Gamma=\{1,\sigma \}$.

An endoscopic data for $G$ is a tuple $(G^{\prime},\mathcal{G^{\prime}},s^{\prime},\xi)$, where
\begin{enumerate}
\item  $G^{\prime}$ is a $K$-group and quasi-split over $\R$, and so has dual Galois automorphism $\sigma_{\widehat{G'}}$.
\item $\mathcal{G}^{\prime}$ is a split extension of $W_{\R}$ by $\widehat{G^{\prime}}$, where $W_{\R}$ acts through $W_{\R}\to\Gamma$, and $\sigma$ act as $\sigma_{\widehat{G^{\prime}}}$ up to an inner automorphism of $\widehat{G^{\prime}}$.
\item $s^{\prime}$ is a semisimple element of $\widehat{G}$,
\item $\xi^{\prime}:\mathcal{G}^{\prime}\to {}^{L}G$ is an embedding of extensions under which the image of $\widehat{G^{\prime}}$ is the identity component of $\Cent(s^{\prime},\widehat{G})$, and the full image lies in $\Cent(s^{''},{}^{L}G)$, for some $s^{''}$ congruent to $s^{\prime}$ modulo the center $Z(\widehat{G})$ of $\widehat{G}$.
\end{enumerate}

We denote $\mathcal{E}(G)$ for the set of equivalence of endoscopic data for $G$. Then $\mathcal{E}(G)=\coprod_{\{M\}}(\mathcal{E}_{\elll}(M)/W(M))$ or $\mathcal{E}(G)=(\coprod_{G^{\prime}\in\mathcal{E}_{\elll}}\mathcal{L}^{G^{\prime}})/\sim$,
 where the equivalence relation is defined by $\widehat{G}$ conjugacy.

\bigskip
We first recall the geometric side of stable local trace formula, which is given by Arthur \cite[\S 6]{A11}. We shall identify $G^{\prime}$ with the diagonal endoscopic datum $G^{\prime}_{V}=G^{\prime}\times\bar{G^{\prime}}$ for $G_{V}=G\times G$, where $G^{\prime}$ represents the datum $(G^{\prime},\mathcal{G}^{\prime},s^{\prime},\xi^{\prime})$, and $\bar{G}^{\prime}$ represents the adjoint datum $(G^{\prime},\mathcal{G}^{\prime},(s^{\prime})^{-1},\xi^{\prime})$. The Langlands-Shelstad transfer factors attached to $(G,G^{\prime})$ depend on an auxiliary data $\widetilde{G^{\prime}}$ for $G^{\prime}$, such that
$\widetilde{\xi^{\prime}}:\mathcal{G}^{\prime}\to {}^{L}\widetilde{G^{\prime}}$ for $L$-morphism. We can choose a compatible auxiliary data for $\bar{G^{\prime}}$, so that the relative transfer factor for $(G,\bar{G^{\prime}})$ is the inverse of the relative transfer factor for $(G,G^{\prime})$, then we have transfer property, $\bar{f_{2}}^{\bar{G^{\prime}}}=\overline{f_{2}^{G^{\prime}}}$. The transfer mappings \cite{A8} were used to construct supplementary linear forms $I^{\mathcal{E}}(f)$ and $S^{G}(f)$  from $I(f)$, where
\[ I^{\mathcal{E}}(f)=\sum_{G^{\prime}\in\mathcal{E}^{0}_{\elll}(G)}\iota(G,G^{\prime})\widehat{S^{\prime}}
(f^{\prime})+\varepsilon(G)S^{G}(f),\]
 the linear forms $\widehat{S^{\prime}}=\widehat{S}^{\widetilde{G^{\prime}}}$ on $SIC(\widetilde{G^{\prime}_{V}},\widetilde{\zeta^{\prime}_{V}})$ are determined inductively by the further requirement such that $I^{\mathcal{E}}(f)=I(f)$ if $G$ is a quasi-split group.
Here \[\iota(G,G^{\prime})=|\Out_{G}(G^{\prime})|^{-1}|Z(\widehat{G^{\prime}})^{\Gamma}/Z(\widehat{G})^{\Gamma}|^{-1}, \] $\mathcal{E}^{0}_{\elll}(G)=\mathcal{E}_{\elll} \backslash \{G\},$ and
\[\varepsilon(G)=\begin{cases}1&\text{if $G$ is quasi-split},\\
0 &\text{otherwise}.
\end{cases}\]

\bigskip
 For $G$ general, we have $I(f)=I^{\mathcal{E}}(f)$ \cite[\S 6]{A11}. If $G$ is quasi-split, we have a geometric expansion \cite[(10.11)]{A11}
\[S^{G}(f)=\sum_{M\in\mathcal{L}}|W^{M}_{0}||W^{G}_{0}|^{-1}(-1)^{\dimm(A_{M}/A_{G})}\int_{\Delta_{G-\reg,\elll}(\widetilde{M^{\prime}},V,\widetilde{\zeta^{\prime}})}n(\delta)^{-1}S^{G}_{M}(\delta,f)\,d\delta.\]
If $f=f_{1}\times\overline{f_{2}}$, $f_{1}\in C_{\cusp}(G,\zeta)$, $f_{2}\in C(G,\zeta)$,
 then \[ S^{G}_{M}(\delta,f_{1}\times\bar{f_{2}})=S^{G}_{M}(\delta,f_{1})\times \overline{f^{G}_{2}(\delta)}.\]

The spectral side of the local trace formula was stabilized in \cite{A8}, when the test function $f\in C_{\cusp}(G,\zeta)\times C(G,\zeta^{-1})$. But it is just a formal formula that matches the geometric side. The point of the present work is to directly construct the spectral side of the stable local trace formula. We need to review a few facts about the spectral side of invariant local trace formula, before constructing a stabilization.

\bigskip
The irreducible tempered representations could well be regarded as the objects dual to semisimple conjugacy classes in $G(\R)$. It is better to take the family of virtual characters \cite{A5}, which is parametrized by a set $T(G,\zeta)$.

\begin{definition}
$T(G,\zeta)$ is the set of $W_{0}$ orbits of the triplet  $\tau=(M,\pi,r), M\in\mathcal{L},\pi\in\Pi_{2}(M,\zeta),r\in\widetilde{R_{\pi}}$, where $\Pi_{2}(M,\zeta)$ stands for the equivalence classes of irreducible unitary representations of $M(\R)$ which are square integrable modulo the center, whose central character is $\zeta$, and $\widetilde{R_{\pi}}$ is a fixed central extension
 $$ 1 \longrightarrow  \widetilde{Z_{\pi}} \longrightarrow \widetilde{R_{\pi}} \longrightarrow R_{\pi}\longrightarrow 1$$ of the $R$-group of $\pi$.
 \end{definition}
The purpose of the extension is to ensure that the normalized intertwining operators, $r\mapsto\widetilde{R_{P}}(r,\pi), r\in\widetilde{R_{\pi}}, P\in\mathcal{P}(M)$ for the induced representation $I_{P}(\pi)$, give a representation of $\widetilde{R_{\pi}}$ instead of just a projective representation of $R_{\pi}$. However, in the Archimedean case, $R_{\pi}$ is a product of groups $\Z/2\Z$, the cocycle which defines $\widetilde{R_{\pi}}$ splits, so we take $\widetilde{R_{\pi}}=R_{\pi}$.

There is a bijection $\rho\mapsto\pi_{\rho}$ from $\Pi(R_{\pi})$ the set of irreducible representations of $R_{\pi}$ onto the set of irreducible constituents of $I_{P}(\pi)$, with the properties that
\begin{equation}
\Theta(\tau,f)=\tr(R_{P}(r,\pi)I_{P}(\pi,f))=\sum_{\rho\in\Pi(R_{\pi})}\tr(\rho^{\vee}(r))\tr(\pi_{\rho}(f)),
\end{equation}
and
\begin{equation}\label{eq:triv}
\tr(\pi_{\rho}(f))=|R_{\pi}|^{-1}\sum_{r\in R_{\pi}}\tr(\rho(r))\tr(R_{P}(r,\pi)I_{P}(\pi,f)).
\end{equation}

\bigskip
We can write \[T(G,\zeta)=\coprod_{\{M\}}(T_{\elll}(M,\zeta)/W(M)),\]
where $T_{\elll}(M,\zeta)=\{\tau=(M',\pi',r)\in T(M,\zeta):\mathfrak{a}^{r}_{M'}=\mathfrak{a}_{M},r\in R_{\pi}\}$. $\{M\}$ as usual runs over the orbits in $\mathcal{L}/W^{G}_{0}$. If $\tau=(M,\pi,r)$ is any triplet, the isotropy subspace $\mathfrak{a}^{r}_{M}$ of $\mathfrak{a}_{M}$ equals $\mathfrak{a}_{L}$ for some $L$. There is an action
\[\tau\mapsto\tau_{\lambda}=(M,\pi_{\lambda},r),\tau\in T_{ell}(L,\zeta),\lambda\in i\mathfrak{a}^{\ast}_{L,Z}\]
of $i\mathfrak{a}^{\ast}_{L,Z}$ on $T_{\elll}(L,\zeta)$, where $\pi_{\lambda}(x)=\pi(x)e^{\lambda(H_{M}(x))}$ for any $x\in M(\R)$.
 This gives the structure for $T(G,\zeta)$ a disjoint union of finite quotients of compact tori. We will write $T(G,\zeta)_{\C}$ to be the disjoint union over $\{M\}$ of the spaces of $W(M)$ orbits in $T_{\elll}(M,\zeta)_{\C}=\{\tau_{\lambda}:\tau\in T_{\elll}(M,\zeta),\lambda\in \mathfrak{a}^{\ast}_{M,Z,\C}\}$, where $\mathfrak{a}^{\ast}_{M,Z}$ is the subspace of linear forms on $\mathfrak{a}_{M}$ which are trivial on the image of $\mathfrak{a}_{Z}$ in $\mathfrak{a}_{M}$, and $\mathfrak{a}^{\ast}_{M,\C}=X(M)_{\R}\otimes \C$,  where $X(M)$ is the set of rational characters on $M$.

\bigskip
 We denote $T_{\disc}(G,\zeta)$ for a set of orbits $(M,\pi,r)$ in $T(G,\zeta)$ such that $W_{\pi}(r)_{\reg}$ is not empty. Then
 \[ T_{\elll}(G,\zeta)\subset T_{\disc}(G,\zeta) \subset T(G,\zeta). \]

 Let $I(G(\R),\zeta)$ be the space of functions
 \[ \alpha :T(G,\zeta)\to \C,\]
which satisfy the following three conditions;
\begin{enumerate}
\item  $\alpha$ is supported on finitely many components of $T(G,\zeta)$,
\item $\alpha$ is symmetric under $W^{G}_{0}$,
\item $\alpha\in S(T(G,\zeta))$.
\end{enumerate}
Here $S(T(G,\zeta))$ is the space of smooth functions $\alpha$ on $T(G,\zeta)$, such that for each $M\in\mathcal{L}$, each integer $n$ and each invariant differential operator $D=D_{\lambda}$ on $i\mathfrak{a}^{\ast}_{M,Z}$  transferred in the obvious way $D_{\tau}\alpha(\tau)=lim_{\lambda\to 0}D_{\lambda}\alpha(\tau_{\lambda}), \tau\in T_{\elll}(M,\zeta)$ to $T_{\elll}(M,\zeta)$, and such that the semi-norm
\[\|\alpha\|_{M,D,n}=Sup_{\tau\in T_{\elll}(M,\zeta)}(|D_{\tau}\alpha(\tau)|(1+\|\mu_{\tau}\|)^{n})\]
is finite, where $\mu_{\tau}=\mu_{\pi}$ for $\tau=(M,\pi,r)$.
$\mu_{\pi}$ is the linear form determined by the infinitesimal character of $\pi$. Then there is a natural topology which makes $I(G(\R),\zeta)$ into a complete topological vector space. By means of the inversion formula \eqref{eq:triv}, we can identify $I(G(\R),\zeta)$ with the topological vector space of functions on $\Pi_{\temp}(G(\R),\zeta)$, and also denoted by $I(G(\R),\zeta)$.

\bigskip
  The trace Paley-Wiener theorem \cite{A6} is equivalent to the assertion that the map which sends $f\in C(G(\R),\zeta)$ to the function $f_{G}(\tau)=\Theta(\tau,f)$ is an open, continuous and surjective linear transformation from $C(G(\R),\zeta)$ onto $I(G(\R),\zeta)$. Observe that if $\tau\in T_{\elll}(G,\zeta)$, there is a function $f\in C(G(\R),\zeta)$ with $f_{G}(\tau)=1$, and such that $f_{G}$ vanishes away from the $i\mathfrak{a}^{\ast}_{G,Z}$ orbit of $\tau\in T(G,\zeta)$. We call such a function $f$  for a pseudo-coefficient of $\tau$.

\bigskip
Stabilization of the spectral side of local trace formula depends on Shelstad's works. She \cite{S2}, \cite{S3} directly constructs the spectral transfer factors and obtains the transfer theorem. The adjoint relation on $K$-groups and the structure of tempered $L$-packets are given in \cite{S4}.

\bigskip
  We denote by $\Pi_{\temp}(G,\zeta)$ the set of tempered representations, with central character for $\zeta$. We have $\Pi_{\temp}(G,\zeta)=\coprod_{\{M\}}\Pi_{2}(M,\zeta)/W(M)$, here $\{M\}$ for the set of $W^{G}_{0}$-orbits of Levi subgroups of $G$.

 The Langlands parameter $\phi: W_{\R}\to {}^{L}G$ is an $L$-homomorphism, which maps from $W_{\R}$ into the $L$-group ${}^{L}G$. We denote $\Phi(G)$ for the set of $\widehat{G}$-orbits of parameters which are tempered, which means that the image of $W_{\R}$ in $\widehat{G}$ is bounded. We denote $\Phi_{2}(G)$ for the subset of parameters in $\Phi(G)$ which are cuspidal. The cuspidal condition means that the image of $W_{\R}$ is contained in no proper parabolic subgroup. There is a canonical decomposition
\[ \Phi(G)= \coprod_{\{M\}}(\Phi_{2}(M)/W(M)).\]

 For any $\phi$, we denote $S_{\phi}$ as the centralizer of the image of $\phi$ in $\widehat{G}$, and $\mathcal{S}_{\phi}$ stands for the group of connected components in $\bar{S}_{\phi}=S_{\phi}/Z(\widehat{G})^{\Gamma}$. We say that $\phi$ is elliptic, if $\bar{S}_{\phi,s}$ is finite for some semisimple element $s\in \bar{S}_{\phi}$. For any parameter $\phi\in\Phi(G)$, we denote the central character $\zeta$ for $\phi$, whose Langlands parameter is just the composition
 \[ W_{F}\xrightarrow{\phi} {}^{L}G\rightarrow {}^{L}Z .\]

 The entire set $\Phi(G)$ decomposes into a disjoint union of the subsets $\Phi(G,\zeta)$. The set $\Phi_{2}(G,\zeta)$ also comes with an action $\phi\mapsto\phi_{\lambda}=\phi\circ\rho_{\lambda}$ of $i\mathfrak{a}^{\ast}_{G,Z}$, where $\rho_{\lambda}\in H^{1}(W_{\R},Z(\widehat{G})^{\Gamma})$, which corresponds to the character $\pi_{\rho_{\lambda}}(x)=\mathnormal{e}^{\lambda(H_{G}(x))}$.

\bigskip
 Given $s\in \bar{S}_{\phi}$, we attach an endoscopic data $G^{\prime}=G^{s}=(G^{s},\mathcal{G}^{s},s,\xi^{s})$. Where $\mathcal{G}^{s}$ is the subgroup of ${}^{L}G$ generated by $\Cent(s,\widehat{G})^{\circ}$ and the image of $\phi$. $\xi^{s}$ is the inclusion $\mathcal{G}^{s}\hookrightarrow{}^{L}G$ and $G^{s}$ is a quasi-split group. Usually, $\mathcal{G}^{s}$ need not be an $L$-group, namely that there might not be an $L$-isomorphism from $\mathcal{G}^{s}$ to ${}^{L}G^{s}$ which is the identity component of $\widehat{G^{\prime}}$. To deal with this problem, we need to make a $z$-extension $\widetilde{G^{\prime}}$ of $G^{\prime}$. For simplicity, we assume $\widetilde{G^{\prime}}=G^{\prime}$, for any $G^{\prime}$.

\bigskip
 Shelstad established the spectral transfer mapping, which is given by a linear combination
\begin{equation}\label{eq:tff}
f^{\prime}(\phi^{\prime})= \sum_{\pi\in\Pi_{\temp}(G(\R))}\Delta(\phi^{\prime},\pi)f_{G}(\pi)
\end{equation}
of irreducible tempered characters $f_{G}(\pi)=\tr(\pi(f)), \pi\in\Pi_{\temp}(G,\zeta),$ on $G(\R)$. The coefficients are spectral transfer factors $\Delta(\phi^{\prime},\pi)$. They are established explicitly by Shelstad in \cite{S3}, which are compatible with the geometric transfer factors. We assume implicitly that the Langlands parameter $\phi$ is relevant to $G$, in the sense that if its image is contained in a parabolic subgroup ${}^{L}P\subset {}^{L}{G}$, then ${}^{L}P$ is dual to a $\Q$-rational parabolic subgroup $P\subset G$. It then gives rise to the $L$-packet $\Pi_{\phi}$ that was an integral part of Langlands's classification of representations of real groups \cite{L2}. $\Pi_{\phi}$ is a finite subset of representations in $\Pi_{\temp}(G,\zeta)$ whose constituents have the same local $L$-functions and $\varepsilon$-factors, and $\Pi_{\temp}(G,\zeta)$ is a disjoint union over $\phi\in \Phi(G,\zeta)$ of the subsets $\Pi_{\phi}$. Shelstad established that for any $\phi$, the distribution
\[ f^{G}(\phi)=\sum_{\pi\in\Pi_{\phi}}f_{G}(\pi)\]
is stable, in the sense that it depends only on the image $f^{G}$ of $f$ in $S(G,\zeta)$. Here $S(G,\zeta)=\{f^{G}:f\in C(G,\zeta)\}$, and $f^{G}$ is the stable orbital integral
\[f^{G}(\delta)=|D(\delta)|^{1/2}\int_{G_{\delta}(\R)\backslash G(\R)} f(x^{-1}\delta x)\,dx =\sum_{\gamma\to\delta}f_{G}(\gamma).\]

When applied to $G^{\prime}$ instead of $G$, this gives the left hand side of \eqref{eq:tff}. We also assume that the given pair $(G^{\prime},\phi^{\prime})$ is relevant to $G$, in the sense that the composite Langlands parameter $\phi=\xi^{\prime}\circ \phi^{\prime}:W_{\R}\to{}^{L}G$ is relevant to $G$. Then we have a bijection mapping
\begin{equation} \label{eq:tfdt}
(G^{\prime},\phi^{\prime})\to(\phi,s),
\end{equation}
where $s\in \bar{S}_{\phi}$.

\bigskip
Shelstad also established the inversion of the transfer mapping, which is given by a linear combination \[f_{G}(\pi)=\sum_{s_{\ssc}}\Delta(\pi,\phi^{s})f^{\prime}(\phi^{s})\]
of stable characters on endoscopic groups, where $s_{\ssc}\in \widetilde{\mathcal{S}_{\phi}}$ comes from an extension
\[ 1\rightarrow\widehat{Z}_{\ssc}\rightarrow\widetilde{\mathcal{S}_{\phi}}=\pi_{0}(S_{\phi,\ssc}) \rightarrow\mathcal{S_{\phi}} \rightarrow 1,\]
where $S_{\phi,\ssc}$ is the preimage of $\bar{S}_{\phi}$ in $\widehat{G}_{\ssc}$, the simply connected cover the derived group of $\widehat{G}$ and $\widehat{Z}_{\ssc}=Z(\widehat{G}_{\ssc})$, $\phi^{s}$ stands for the parameter $\phi^{\prime}$ that corresponds to the pair $(\phi,s)$ under the map \eqref{eq:tfdt}. The inversion of transfer mapping rests on explicit adjoint relations for spectral transfer factors $\Delta_{\spec}$ defined initially as a product $\Delta_{I}\Delta_{II}\Delta_{III}$ in the $G$-regular case in \cite{S3}, and we have adjoint relations
\[ \sum_{s_{\ssc}} \Delta(\pi,\phi^{s})\Delta(\phi^{s},\pi^{\prime})=\delta(\pi,\pi^{\prime}),\]
where the sum is over semisimple representative $s_{\ssc}$ for \[\widetilde{\mathcal{S}_{\mathcal{\phi}}}/\Ker(\widetilde{\mathcal{S}_{\mathcal{\phi}}}\to\mathcal{S}_{\phi})\simeq\mathcal{S}_{\phi},\] $\delta(.,.)$ denotes the Kronecker delta function. If $\pi$ and $\phi^{s}$ correspond to the same $L$-parameter $\phi$, then \[\Delta(\pi,\phi^{s})=\frac{1}{n(\pi)}\Delta(\phi^{s},\pi)^{-1},\]
and \[\Delta(\phi^{s},\pi)^{-1}=\overline{\Delta(\phi^{s},\pi)}/\|\Delta\|^{2},\] $n(\pi)=|\mathcal{S_{\phi}}|$
is the cardinality of the $L$-packet of $\pi$, $\|\Delta\|=|\Delta(\phi^{s},\pi)|$ is a constant (independent s), which is compatible with geometric transfer factor. However, the geometric transfer factor is unitary, so we take $\|\Delta\|=1$.

\bigskip
We also obtain the other adjoint relations
\[ \sum_{\pi\in \Pi_{\phi}} \Delta(\phi^{s},\pi)\Delta(\pi,\phi^{s^{\prime}})=\delta(s,s^{\prime}),\]
where $\phi=\xi^{s}\circ\phi^{s}$,  $s,s^{\prime}$ in $\mathcal{S}_{\phi}$.

So we can define an inversion adjoint transfer factor
 \[\Delta(\tau,\phi^{s})=\sum_{\chi\in \widehat{R}_{\pi}}\overline{\chi(r)}\Delta(\pi^{\chi},\phi^{s}),\]
 where $\overline{\chi(r)}=\tr(\rho^{\vee}(r))$.
(Recall that $R_{\pi}$ is finite abelian group.)

 We obtain
 \begin{equation} \label{eq:ivstf}
 \Theta(\tau,f)=\sum_{s\in\mathcal{S}_{\phi}}\Delta(\tau,\phi^{s})f^{\prime}(\phi^{s}).
 \end{equation}

\bigskip
However, when we define the transfer factor
$\Delta(\phi^{s},\tau)$, we need to assume that the Langlands parameter $\phi$ to be elliptic. This means that $\Pi_{\phi}$ contains elliptic representations. $\phi$ factors through a discrete parameter for a cuspidal Levi subgroup ${}^{L}M$, and so through ${}^{L}T_{M}$ \cite{L2}, where $T_{M}$ is the maximal torus which is compact modulo the center of $M$.  We consider the associated short exact sequence \cite{S3}
$$1\rightarrow\mathcal{E}(T_{M})\rightarrow\mathcal{ S}_{\phi}\rightarrow R_{\phi}\rightarrow 1,$$
where $R_\phi$ is the Langlands $R$-group, and the group $\mathcal{E}(T_{M})$ is isomorphic to $\mathcal{S}_{\phi_{M}}$, where $M$ is a Levi subgroup of $G$, and $\phi_{M}:W_{\R}\to {}^{L}M$ is a Langlands parameter for $M$ whose image in ${}^{L}G$ equals $\phi$, and whose L-packet $\Pi_{\phi_{M}}$ consists of representations in the discrete series of $M(\R)$.

\bigskip
Suppose $\pi_{M}\in\Pi_{\phi_{M}}$ corresponds to the character $\chi$ on the group $\mathcal{S}_{\phi_{M}}$. Since $\mathcal{S}_{\phi}$ is an abelian group, $R_{\chi}$ equals the full group $R_{\phi}$, where $R_{\chi}$ is the subgroup of elements in $R_{\phi}$ that stabilize $\chi$, and $\chi$ extends to a character $\theta$ on $\mathcal{S}_{\phi}$. The set of such extension $\theta$ is a torsor under the action of the characters in $R_{\phi}$. It corresponds to subset $\Pi_{\phi,\pi_{M}}$ of $\Pi_{\phi}$, composed of the irreducible constituents of the induced representation $I^{G}_{P}(\pi_{M})$, where $P$ belongs to the set $\mathcal{P}(M)$ of parabolic subgroups of $G$ with Levi component $M$. So we have $|R_{\pi_{M}}|=|R_{\phi}|$.

\bigskip
 On the other hand, we identity the stabilizer $W_{\phi}$ of $\phi_{M}$ with a subgroup of $W(M)$, where $W_{\phi}=W(S_{\phi},A_{\widehat{M}})$, which is to say, the group of automorphisms of $A_{\widehat{M}}$ induced from $\bar{S}_{\phi}$, and $A_{\widehat{M}}=(Z(\widehat{M})^{\Gamma})^{\circ}$. This Weyl group contains the stabilizer $W_{\pi_{M}}$ of $\pi_{M}$. It is a consequence of the disjointness of tempered $L$-packets for $M$ that $W_{\phi}$ contains $W_{\pi_{M}}$. From the above discussion, we know that any element in $W_{\phi}$ stabilizes $\pi_{M}$. Therefore $W_{\phi}$ equals $W_{\pi_{M}}$. Moreover, we know that elements in the subgroup $W^{0}_{\phi}$ of $W_{\phi}$ give scalar intertwining operators for the induced representation $I^{G}_{P}(\pi_{M})$, where $W^{\circ}_{\phi}=W(S^{\circ}_{\phi},A_{\widehat{M}})$ is to the normal subgroup of automorphisms in $W_{\phi}$ that are induced from the connected component $\overline{S}^{\circ}_{\phi}$. It follows that $W^{0}_{\phi}$ is contained in the subgroup $W^{0}_{\pi_{M}}$ of $W_{\pi_{M}}$. We have a surjective mapping
\[R_{\phi}\cong W_{\phi}/W^{0}_{\phi}\mapsto W_{\pi_{M}}/W^{0}_{\pi_{M}}\cong R_{\pi_{M}}, \quad \pi_{M}\in\Pi_{\phi_{M}},\]
 for any elliptic parameter $\phi\in \Phi_{\elll}(G,\xi)$. So we obtain $R_{\pi_{M}}=R_{\phi}, W^{\circ}_{\pi_{M}}=W^{\circ}_{\phi},$ and $W_{\pi_{M}}=W_{\phi}$. We define the subset $T_{\phi}=\{(M,\pi,r):M= M_{\phi}, \pi\in \Pi_{\phi_{M}}, r\in R_{\phi}\}$ of $T(G)$, then $|T_{\phi}|=|\Pi_{\phi_{M}}||R_{\phi}|=|\mathcal{S}_{\phi}|=|\Pi_{\phi}|$. And we get a bijection from $T_{\phi}$ to $\mathcal{S}_{\phi}$.

\bigskip
 We can define the adjoint transfer factor
 \[ \Delta(\phi^{s},\tau)=\sum_{\chi\in\widehat{R}_{\pi}}\frac{1}{|R_{\pi}|}\chi(r)\Delta(\phi^{s},\pi^{\chi}),\]
where $\tau=(M,\pi,r)$,  $R_{\pi}$ is a $2$-group, and $\pi^{\chi}$ is the irreducible component of the induced representation of $\pi$, which corresponds to $\chi$ by Arthur's classification Theorem in \cite[\S 2]{A5}.
 So we obtain
 \begin{equation} \label{eq: stabtf}
 f^{\prime}(\phi^{s})=\sum_{\tau\in T_{\phi}} \Delta(\phi^{s},\tau)\Theta(\tau,f).
 \end{equation}

 \bigskip
The transfer factors have the following properties.
\begin{proposition} \label{prop:trsfr}
If $\phi$ is elliptic, then
\begin{enumerate}
\item \[\Delta(\tau,\phi^{s})=\frac{|R_{\phi}|}{|\mathcal{S}_{\phi}|}\overline{\Delta(\phi^{s},\tau)},\]
\item we have the adjoint relations,
\begin{align}
 \sum_{\tau\in T_{\phi}}\Delta(\phi^{s_{1}},\tau)\Delta(\tau,\phi^{s_{2}})&= \delta(\phi^{s_{1}},\phi^{s_{2}}), \label{equation 5.21}\\
 \sum_{s\in \mathcal{S}_{\phi}}\Delta(\tau,\phi^{s})\Delta(\phi^{s},\tau_{1}) &= \delta(\tau,\tau_{1}).\label{equation 5.22}
\end{align}
\end{enumerate}
\end{proposition}
\begin{proof}
 We first check (1).
\[ \Delta(\tau,\phi^{s})=\sum_{\chi\in \widehat{R}_{\pi}}\overline{\chi(r)}\Delta(\pi^{\chi},\phi^{s}),\]
and
\[\overline{\Delta(\phi^{s},\tau)}=\sum_{\chi\in\widehat{R}_{\pi}}\frac{1}{|R_{\pi}|}\overline{\chi(r)}\overline{\Delta(\phi^{s},\pi^{\chi})}=\sum_{\chi\in\widehat{R}_{\pi}}\frac{n(\pi)}{|R_{\pi}|}\overline{\chi(r)}\Delta(\pi^{\chi},\phi^{s}),\]
where $|R_{\pi}|=|R_{\phi}|$, and $n(\pi)=|\mathcal{S}_{\phi}|$. We obtain the equation (1).

We now prove (2). We only check (\ref{equation 5.21}), as (\ref{equation 5.22}) is similar. Observe that $\phi$ is relevant to $G$,  and \[  f^{\prime}(\phi^{s_{1}})=\sum_{\tau\in T_{\phi}}\Delta(\phi^{s_{1}},\tau)\Theta(\tau,f)=\sum_{s\in\mathcal{S}_{\phi}}\sum_{\tau\in T_{\phi}}\Delta(\phi^{s_{1}},\tau)\Delta(\tau,\phi^{s})f^{\prime}(\phi^{s}).\]
Using the fact that the characters of representations are linear independent, then we obtain the identity (\ref{equation 5.21}).
\end{proof}

\bigskip
 We return to analyze the transfer factors
 \[ \Delta(\phi^{s},\tau)=\sum_{\chi\in\widehat{R}_{\pi}}\frac{1}{|R_{\pi}|}\chi(r)\Delta(\phi^{s},\pi^{\chi}),\]
 where $\Delta(\phi^{s},\pi^{\chi})$ is defined directly by Shelstad in \cite{S3}, and extended to a $K$-group \cite{S4}. She also checked Arthur's conjecture for the transfer factors in \cite{A12}, and obtained the formula for $\Delta(\phi^{s},\pi^{\chi})$ in \cite{S4}:
\[\Delta(\phi^{s},\pi^{\chi})=\rho(\Delta,s_{\ssc})\langle {s_{\ssc}},{\pi^{\chi}}\rangle,\]
where $\rho(\Delta,s_{\ssc})$ satisfies       $\rho(t\Delta,z_{\ssc}s_{\ssc})=t\rho(\Delta,s_{\ssc})\zeta_{G}(z_{\ssc})^{-1}$ for $t\in \C^{\times}$, $z_{\ssc}\in Z(\widehat{G}_{\ssc}$), and $s_{\ssc}$ is the preimage $s$ of the mapping $\widetilde{\mathcal{S}}_{\phi}\to \mathcal{S}_{\phi}$. $\zeta_{G}$ comes from Arthur's paper \cite{A8}. And
\[\rho(\Delta,s_{\ssc})=\zeta_{G}(s_{\ssc})^{-1}\delta(\pi^{s},\pi^{base}),  \quad \zeta_{G}(s_{\ssc})=\langle{s_{\ssc}},\pi^{base}\rangle.\]

If $G$ is quasi-split, $\rho(\Delta,s_{\ssc})=1$, so the formula for the transfer factor simplifies: \[\Delta(\phi^{s},\pi^{\chi})=\langle{s_{\ssc}},{\pi^{\chi}}\rangle.\] In section \ref{sec:se}, we will stabilize the spectral side of the invariant local trace formula when one of the component of test function is cuspidal, which is enough to give the multiplicity formula.

\bigskip
\section{Stabilization of the elliptic terms} \label{sec:se}

 We now consider the test function $f = f_1 \times \bar{f_2}$, $f_1 \in C_{\cusp}(G(\R),\zeta)$, $f_2 \in C(G(\R),\zeta)$, $G$ is a $K$-group over $\R$.
Recall that $C_{\cusp}(G(\R),\zeta)$ stands for the space of functions $f_1$ in $C(G(\R),\zeta)$ that are cuspidal, in the sense that the orbital integral
\[\gamma \mapsto f_{1,G}(\gamma) = J_G(\gamma,f_{1}),     \quad  \gamma \in \Gamma(G)\]
is supported on the subset $\Gamma_{\elll}(G)$ of elliptic classes in $\Gamma(G)$.

\bigskip
We assume that $f_1$ is cuspidal. We set
\[I_{\disc}(f) = \int_{T_{\elll}(G, \zeta)} i^G(\tau)f_{1,G}(\tau)\overline{f_{2,G}(\tau)}d\tau,\]
where $i^G(\tau) = |R_{\pi,r}|^{-1}|\de(1-r)_{\mathfrak{a}_M/\mathfrak{a}_G}|^{-1}$, $\tau = (M, \pi, r)$, $T_{\elll}(G,\zeta) = \prod_{\alpha \in \pi_0(G)} T_{\elll} (G_{\alpha}, \zeta_{\alpha})$. \\
For the given $f$, $I_{\disc}(f)$ equals the spectral side of the local trace formula. We have
\[I(f) = I_{\disc}(f). \]

\bigskip
Now we can regard $I_{\disc}$ as a linear form on the subspace
\[C_{1-\cusp} (G_{V}, \zeta_{V}) = C_{\cusp} (G,\zeta)\otimes C(G, \zeta^{-1}) \]
of $C(G_{V}, \zeta_{V})$.\\
If $f_1 \in C_{\cusp}(G, \zeta)$, then $f_1$ is supported on $T_{\elll}(G, \zeta)$. For stabilization, we need to consider Langlands parameters $\phi$ which are elliptic. We can define a corresponding set
 \[\Phi_{\elll}^{\mathcal{E}}(G,\zeta) = \{(G^{\prime}, \phi^{\prime}): G^{\prime} \in \mathcal{E}_{\elll}(G), \phi^{\prime} \in \Phi_2(G^{\prime}, G,\zeta) \}, \]
where $\Phi_2(G^{\prime}, G, \zeta) = \Phi_2(G, \zeta)/\Out_{G}(G^{\prime})$, $\Out_G(G^{\prime}) = \Aut_G(G^{\prime})/\xi^{\prime}(\widehat{G^{\prime}})$, and
\[\Aut_G(G^{\prime}) = \{ g\in \widehat{G}: gs^{\prime}g^{-1} \in s^{\prime}Z(\widehat{G}), g\mathcal{G}^{\prime}g^{-1} = \mathcal{G}^{\prime} \}.\]
We denote $SI_{\cusp}(G^{\prime},\zeta)$ for the set of linear forms $f^{\prime}(\phi^{\prime})$ on $\Phi_{2}(G^{\prime}, \zeta)$ obtained from the transfer map. $f^{\prime}(\phi^{\prime})$ depends only on the image of $\phi^{\prime}$ in $\Phi_{2}(G^{\prime}, G, \zeta)$, which is the set of $\Out_{G}(G^{\prime})$-orbit in $\Phi_{2}(G^{\prime}, \zeta)$ \cite{S4}. If $f^{\prime}(\phi^{\prime}) \in SI_{\cusp}(G^{\prime},\zeta)$, then $f^{\prime}(\phi^{\prime})$ is supported on $\Phi_{2}(G^{\prime}, G, \zeta)$.

\bigskip
We have
\begin{align*}
\theta(\tau,f) &=\sum_{\phi^{\prime}\in \Phi_{\elll}^{\mathcal{E}}(G)} \Delta(\tau, \phi^{\prime}) f^{\mathcal{E}}(\phi^{\prime})\\
&=\sum_{G^{\prime}\in \mathcal{E}_{\elll}(G)} \sum_{\phi^{\prime} \in \Phi_2(G^{\prime}, G, \zeta)} \Delta(\tau, \phi^{\prime}) f^{G^{\prime}}(\phi^{\prime}),
\end{align*}
and \[f^{G^{\prime}}(\phi^{\prime})=\sum_{\tau\in T_{\elll}(G,\zeta)}\Delta(\phi^{\prime},\tau)\theta(\tau,f).\]
\begin{lemma}
The transfer factors $\Delta(\tau, \phi)$ and $\Delta(\phi, \tau)$  have finite support in $\phi$ for fixed $\tau$, and finite support in $\tau$ for fixed $\phi$. Moreover,
\begin{align*}
&\sum_{\phi^{\prime}\in \Phi_{\elll}^{\mathcal{E}}(G,\zeta)} \Delta(\tau, \phi^{\prime}) \Delta(\phi^{\prime}, \tau_1) = \delta(\tau, \tau_1), &&  \tau, \tau_1 \in T_{\elll}(G,\zeta), \\
and &\sum_{\tau\in T_{\elll}(G,\zeta)} \Delta(\phi^{\prime}, \tau) \Delta(\tau, \phi_{1}^{\prime})  = \delta(\phi^{\prime}, \phi_{1}^{\prime}), && \phi^{\prime}, \phi_{1}^{\prime}\in \Phi_{\elll}^{\mathcal{E}}(G, \zeta),
\end{align*}
where $\delta(\tau, \tau_1)$ and $\delta(\phi^{\prime}, \phi_{1}^{\prime})$ are Kronecker delta functions.
\end{lemma}

\begin{proof}
That $\Delta(\tau, \phi)$ and $\Delta(\phi, \tau)$  have finite support, which is equivalent to saying that Shelstad's transfer factor $\Delta(\pi, \phi)$ and $\Delta(\phi, \pi)$  have finite support, following from \cite[\S7]{S3}. The proof of second part is similar to Proposition \ref{prop:trsfr} (2).
\end{proof}

\bigskip
We define a measure on $\Phi_2(G, \zeta)$ by setting
\[\int_{\Phi_2(G, \zeta)} \beta(\phi) d\phi = \sum_{\phi\in \Phi_2(G)/i\mathfrak{a}_{G, Z}^{\ast}} \int_{i\mathfrak{a}_{G, Z}^{\ast}} \beta(\phi_\lambda) d\lambda\]
for any $\beta \in C(\Phi_2(G,\zeta))$. So $\Phi_{\elll}^{\mathcal{E}}(G, \zeta)$  has the measure obtained from the quotient measures on the set $\Phi_2(G', G,\zeta)$ and the transfer factors govern the change of variables of integration.

\begin{lemma}\label{lemma:cmm}
Suppose that $\alpha \in C(T_{\elll}(G), \zeta)$, and that $\beta \in C_{\cusp}(\Phi_{\elll}^{\mathcal{E}}(G),\zeta)$. Then
\begin{align*}
&\int_{T_{\elll}(G, \zeta)} \sum_{\phi\in\Phi_{\elll}^{\mathcal{E}}(G, \zeta)} \beta(\phi) \Delta(\phi, \tau)\alpha(\tau) d\tau \\
=&\int_{\Phi_{\elll}^{\mathcal{E}}(G, \zeta)} \sum_{\tau \in T_{\elll}(G,\zeta)} \beta(\phi) \Delta(\phi, \tau)\alpha(\tau) d\phi.
\end{align*}
\end{lemma}

\begin{proof}
According to the definition of the measure $d\tau$, we can decompose the left hand side of the required identity into an expression
\[\sum_{\tau} \int_{i\mathfrak{a}_{G,Z}^{\ast}} \sum_{\phi}\sum_{\mu} \beta(\phi_{\mu}) \Delta(\phi_{\mu},\tau_\lambda)\alpha(\tau_\lambda)d\lambda .\]
Here $\tau\in T_{\elll}(G,\zeta)/i\mathfrak{a}^{\ast}_{G,Z}$, $\phi\in\Phi^{\mathcal{E}}_{\elll}(G,\zeta)/i\mathfrak{a}^{\ast}_{G,Z},$ $\mu\in i\mathfrak{a}^{\ast}_{G,Z}.$  We recall that the transfer factor $\Delta(\phi, \tau)$ vanishes unless $\tau\in T_{\phi}$. We observe from the definition that $\phi_{\mu} = \phi\circ \rho_{\mu}$.
Here $\rho_{\mu} \in H^{1}(W_{\R}, Z(\widehat{G})^{\Gamma})$ and it corresponds to the $L$-packet which is $\Pi_{\phi_{\mu}} = \{ \pi_{\rho_{\mu}} \otimes \pi | \pi \in \Pi_{\phi}\}$, where $\pi_{\rho_{\mu}}(x) = e^{\mu(H_G(x))}$. However, $\tau_{\lambda} = (M, \pi_{\lambda}, r)$, where $\pi_{\lambda}(x) = \pi(x)e^{\lambda (H_G(x))}$. If $\Delta(\phi_{\mu},\tau_{\lambda})$ doesn't vanish, then $\phi_{\mu}$ and $\tau_{\lambda}$ correspond to the same parameter $\phi_\lambda$. So, $\mu = \lambda$. We see that the sum over $\mu$ reduces to the one element $\mu = \lambda$. The expression becomes
\[\sum_{(\tau, \phi)} \int_{i\mathfrak{a}_{G,Z}^{\ast}} \beta(\phi_{\lambda}) \Delta(\phi_{\lambda}, \tau_{\lambda})\alpha(\tau_{\lambda}) d\lambda,\]
where $(\tau, \phi)$ is summed over pairs in $(T_{\elll}(G, \zeta)\times\Phi_{\elll}^{\mathcal{E}}(G,\zeta))/i\mathfrak{a}_{G,Z}^{\ast}$.
From its obvious symmetry, we conclude that the expression must also be equal to the right hand side of the required identity. The identity is therefore valid.
\end{proof}

\bigskip
In the following, we will stabilize the elliptic terms of the spectral side of invariant local trace formula and obtain the explicit formula for the coefficients .

Assume that $f_1$ is a cuspidal function, we denote $|d(\tau)| = |d(r)| = |\de(1-r)_{\mathfrak{a}_M/\mathfrak{a}_G}|^{-1}$, $i^{G}(\tau)=|R_{\pi,r}|^{-1}|d(\tau)|^{-1}$. We need to stabilize $d(\tau)f_{1,G}(\tau)$.

\begin{lemma} \label{lemma:cf}
If $f_1 \in C_{\cusp}(G(\R), \zeta)$, which is supported on $T_{\elll}(G,\zeta)$, there exists a cuspidal function $\widetilde{f}_{1}\in C(G(R),\zeta)$, satisfying
\[\widetilde{f}_{1,G}(\tau) = d(\tau)f_{1,G}(\tau).\]
\end{lemma}

\begin{proof}
Notice that $\tau \to d(\tau)f_{1,G}(\tau)$ is also a function in $I(G(\R), \zeta)$, which is supported on $T_{\elll}(G,\zeta)$. Applying the trace Paley-Wiener theorem, there exists a cuspidal function $\widetilde{f_1}\in C(G(\R),\zeta)$ such that
\[\widetilde{f}_{1,G}(\tau) = \theta(\tau, \widetilde{f_1}) = d(\tau)f_{1,G}(\tau).\]

\end{proof}

\bigskip
For computing the coefficient of the spectral side of the stable local trace formula, we need to consider two cases. One is that $\pi$ is elliptic as well as regular, the other is that $\pi$ is elliptic and not regular. Consider the first case, then $\pi$ is a discrete series representation. Thus there exists $\phi$ which is cuspidal, and satisfies $\pi\in\Pi_{\phi}$. The image of $\phi$ is contained in no proper parabolic subgroup of ${}^{L}G$. For each $s_{\ssc}$ in $\mathcal{S}^{\ssc}_{\phi}$, the parameter $\phi$ is discrete, so that $\pi^{s}$ is also a discrete series representations of $G^{\prime}=(G^{\prime},s,\mathcal{G}^{\prime},\xi)$.

\bigskip
We need to recall some details of Langlands parameters $\phi$. We fixed a splitting $spl_{\widehat{G}}=(\mathcal{B},\mathcal{T},\{X_{\alpha^{\vee}}\})$ of $\widehat{G}$, where $\mathcal{B}$ is a Borel subgroup in $\widehat{G}$, $\mathcal{T}$ is a maximal torus in $\widehat{G}$ included in $\mathcal{B}$, $X_{\alpha^{\vee}}$ is a eigenvector of $\alpha^{\vee}$. Let $\iota$ be half of the sum of the positive roots of $\mathcal{T}$ in $\mathcal{B}$. Then according to \cite{L2} or \cite[\S7]{S1}, one considers pairs $(\mu,\lambda) \in (X_{\ast}(\mathcal{T})\otimes \mathbb{C})^{2}$, which parameterize Langlands parameters $\phi$, with $\phi=\phi(\mu,\lambda)$ being defined by
\[ \phi(z\times1)=z^{\mu}\bar{z}^{\sigma_{T}(\mu)}\] for $z\in \mathbb{C}^{\ast}$, and \[\phi(1\times\sigma)=e^{2\pi i\lambda}n(\sigma_{T})\times(1\times\sigma),\]
 where $T$ is a maximal torus of $G(\R)$, there is a canonical $\widehat{G}$-conjugacy class of admissible embedding $\xi: {}^{L}T\longrightarrow {}^{L}G$, then $\xi$ maps $\widehat{T}$ to $\mathcal{T}$ by the isomorphism attached to the pair $(\mathcal{B},\mathcal{T})$ and the choice of a Borel subgroup $B$ in $G$ containing $T$. There exists $h\in\widehat{G}$ such that $(h^{-1}\widehat{B}h,h^{-1}\widehat{T}h)=(\mathcal{B},\mathcal{T})$. We define $\omega_{T}(\sigma):=\Int(h^{-1}\sigma(h))$, which is an element in the Weyl group $W(\widehat{G},\mathcal{T})$, then $\sigma_{T}=\omega_{T}(\sigma)\times\sigma$. We have $n(\sigma_{T})=n(\omega_{T}(\sigma))$, and $n(\omega_{T}(\sigma))$ was defined by $\{X_{\alpha^{\vee}}\}$, as in \cite[section 2.1]{LS}. $\mu$ and $\lambda$ satisfy the property
\[\frac{1}{2}(\mu-\sigma_{T}\mu)-\iota+(\lambda+\sigma_{T}\lambda)\in X_{\ast}(\mathcal{T}).\]
Here $\mu$ is determined uniquely, while $\lambda$ is determined modulo \[X_{\ast}(\mathcal{T})+\{\nu-\sigma_{T}\nu: \nu\in X_{\ast}(\mathcal{T})\otimes\mathbb{C}\},\]
 and $\phi$ is determined uniquely up to $\mathcal{T}$-conjugacy.

\bigskip
  For the endoscopic datum $G^{\prime}=(G^{\prime},s, \mathcal{G^{\prime}},\xi^{\prime})$, where $\xi^{\prime}:\mathcal{G^{\prime}}\to {}^{L}G$ is an embedding, and $\xi^{\prime}$ can be parameterized by $(\mu^{\ast},\lambda^{\ast})$. We assume that ${}^{L}G^{\prime}$ equals $\mathcal{G^{\prime}}$, $\phi^{\prime}:W_{\mathbb{R}}\to {}^{L}G^{\prime}$ is a Langlands parameter, and $\phi^{\prime}=\phi^{\prime}(\mu^{\prime},\lambda^{\prime})$, $\phi=\xi^{\prime}\circ\phi^{\prime}$. We then have a relation:
  $ \mu=\mu^{\prime}+\mu^{\ast}$, and $\lambda=\lambda^{\prime}+\lambda^{\ast}$.
  If $\mu$ is dominant and regular, then $\pi$ is a discrete series, which belongs to $\Pi_{\phi}$, and $\pi^{\prime}\in \Pi_{\phi^{\prime}}$ is also a discrete series.
    In this case, we know the cardinal number of the $L$-packet $\Pi_{\phi}$, which equals $|\mathcal{S}_{\phi}|$ in \cite[Corollary 7.6]{S4}, and \[|\mathcal{S}_{\phi}|=|\mathcal{E}(T)|.\]
 In this case the $R$-group $R_{\phi}$ is trivial. Then we obtain
  \[\frac{|\mathcal{S}_{\phi}|}{|\mathcal{S}_{\phi^{\prime}}|}=\frac{|\mathcal{E}(T)|}{|\mathcal{E}(T^{\prime})|}=\frac{|\mathcal{K}_{T}|}{|\mathcal{K}_{T^{\prime}}|}=\frac{|Z(\widehat{G^{\prime}})^{\Gamma}|}{|Z(\widehat{G})^{\Gamma}|} . \]
Here $\mathcal{K}_{T}=\pi_{0}((\widehat{T})^{\Gamma}/Z(\widehat{G})^{\Gamma})$. Since the tori $T$ and $T^{\prime}$ are isomorphic, the third equality holds, and the second equality comes from the Tate-Nakayama duality.

\bigskip
      In the following we deal with the singular elliptic case. In other words, if $\pi\in\Pi_{\phi_{1}}$ is elliptic representation, then the Langlands parameter $\phi_{1}$ factors through a discrete parameter for a cuspidal Levi subgroup ${}^{L}M$ and so through ${}^{L}T_{M}$, where the maximal torus $T_{M}$ is compact modulo the center of $M$ \cite{L2}. Following the argument of \cite{S1}, we can construct a new Langlands parameter $\phi$ in the conjugacy class of $\phi_{1}$ and factoring through ${}^{L}T$, where $T$ is compact modulo the center of $G$. This new Langlands parameter $\phi$ will be of the form $\phi(\mu,\lambda)$ as in the first case, but now the regularity requirement on $\mu$ is not necessary, we call such a $\phi$ for a limit of discrete parameters. We obtain, by transfer of a discrete parameter $\phi^{\prime}$ for an elliptic endoscopy group, a limit of discrete parameters $\phi$ \cite[section 9]{S4}, so we obtain
    \[\frac{|\mathcal{S}_{\phi}|}{|\mathcal{S}_{\phi^{\prime}}|}=\frac{|Z(\widehat{G^{\prime}})^{\Gamma}||\Out_{G}(G^{\prime},\phi^{\prime})|}{|Z(\widehat{G})^{\Gamma}|}=\frac{|Z(\widehat{G^{\prime}})^{\Gamma}|}{|Z(\widehat{G})^{\Gamma}|} . \]
 The second equality comes from the discrete case, where $\Out_{G}(G^{\prime},\phi^{\prime})$ is the stabilizer of $\phi^{\prime}$ (as a $\widehat{G^{\prime}}$-orbit) in the finite group $\Out_{G}(G^{\prime})$. Thus we have the following lemma.

\begin{lemma} \label{lemma:cff}
If $\phi$ is elliptic, then we have the coefficient relation
\[\frac{|S_{\phi}|}{|S_{\phi^{\prime}}|}=|Z(\widehat{G^{\prime}})^{\Gamma}/Z(\widehat{G})^{\Gamma}|.\]
\end{lemma}

\bigskip
\begin{theorem}
If $f = f_1 \times \bar{f_2}$, $f_1 \in C_{\cusp}(G(\R),\zeta)$, $f_2 \in C(G(\R),\zeta)$, then
\begin{align*}
I_{\disc}(f) =&\int_{T_{\elll}(G,\zeta)} i^{G}(\tau)f_{1,G}(\tau)\overline{f_{2,G}(\tau)} d\tau  \\
=&\sum_{G'\in\mathcal{E}_{\elll}(G)} \iota(G, G') \widehat{S}^{G'}(f'),
\end{align*}
and  $\widehat{S}^{G'}(f')$ is a stable distribution on $G^{\prime}$,  where
\begin{align*} i^{G}(\tau) =& |d(\tau)|^{-1}|R_{\pi,r}|^{-1}, \\ \widehat{S}^{G'}(f') =& \int_{\Phi_2(G', \zeta)} S^{G'}(\phi')\widetilde{f_1}'(\phi')\overline{f_{2}'(\phi')} d\phi', \\
S^{G'}(\phi^{\prime}) =& \frac{1}{|\mathcal{S}_{\phi^{\prime}}|}, \qquad \phi = \xi'\circ\phi'.
\end{align*}
\end{theorem}

\begin{proof}
By \eqref{eq:inf} we have
\begin{equation}\label{eq:form}
I_{\disc}(f) = \int_{T_{\elll}(G, \zeta)} |R_{\pi, r}|^{-1} |d(\tau)|^{-1} f_{1,G}(\tau)\overline{f_{2,G}(\tau)}d\tau.
\end{equation}
 Applying Lemma \ref{lemma:cf}, then \eqref{eq:form} equals
\begin{align*}
&\int_{T_{\elll}(G,\zeta)} |R_{\pi, r}|^{-1}\widetilde{f}_{1,G}(\tau)\overline{f_{2,G}(\tau)} d\tau \\
=&\int_{T_{\elll}(G, \zeta)}|R_{\phi}|^{-1}\sum_{\phi\in\Phi_{\elll}^{\mathcal{E}}(G,\zeta)} \Delta(\tau, \phi) \widetilde{f_1}^{\mathcal{E}}(\phi)\overline{f_{2,G}(\tau)} d\tau\\
=&\int_{T_{\elll}(G,\zeta)} |R_{\phi}|^{-1}\sum_{\phi\in \Phi_{\elll}^{\mathcal{E}}(G, \zeta)} \frac{|R_{\phi}|}{|\mathcal{S}_{\phi}|} \widetilde{f_{1}}^{\mathcal{E}}(\phi)\overline{\Delta (\phi, \tau) f_{2,G}(\tau)}d\tau.\\
\end{align*}

Applying Lemma \ref{lemma:cmm}, We see that this last expression can be written as
\[\int_{\Phi_{\elll}^{\mathcal{E}}(G, \zeta)} |\mathcal{S}_{\phi}|^{-1} \widetilde{f_1}^{\mathcal{E}}(\phi) \sum_{\tau\in T_{\elll}(G, \zeta)} \overline{\Delta(\phi, \tau)f_{2,G}(\tau)} d\phi,\]
which is just
\begin{align*}
&\int_{\Phi_{\elll}^{\mathcal{E}}(G,\zeta)}|\mathcal{S}_{\phi}|^{-1} \widetilde{f_1}^{\mathcal{E}}(\phi)\overline{f_{2}^{\mathcal{E}}(\phi)} d\phi \\
=&\sum_{G'\in \mathcal{E}_{\elll}(G)} \int_{\Phi_{2}(G', G, \zeta)} |\mathcal{S}_{\phi}|^{-1} \widetilde{f_1}'(\phi') \overline{f_{2}'(\phi')} d\phi'.\\
\end{align*}

We obtain
\[ I_{\disc}(f) = \sum_{G'\in \mathcal{E}_{\elll}(G)} \iota(G,G') \int_{\Phi_{2}(G', \zeta)} |\mathcal{S}_{\phi}|^{-1} |Z(\widehat{G}')^{\Gamma}/Z(\widehat{G})^{\Gamma}|\widetilde{f_1}'(\phi')\overline{f_{2}'(\phi')} d\phi'.\]
We denote the coefficient as $S^{G'}(\phi') = |\mathcal{S}_{\phi}|^{-1} |Z(\widehat{G}')^{\Gamma}/Z(\widehat{G})^{\Gamma}|=|\mathcal{S}_{\phi^{\prime}}|^{-1}$, whence the last equality by the Lemma \ref{lemma:cff}.\\
Then we obtain the required formula
\[I_{\disc}(f) = \sum_{G'\in\mathcal{E}_{\elll}(G)} \iota(G, G') \widehat{S}^{G'}(f').\]
Since $\widetilde{f_{1}}'(\phi')$, and $f_{2}'(\phi')$ are stable as distributions on $G^{\prime}$, $\widehat{S}^{G'}(f')$ is stable as distributions on $G^{\prime}$. We have thus proved the theorem.
\end{proof}

\bigskip
\section{Characters and Stable orbital integral} \label{sec:cs}

Assume that the test function $f=f_{1}\times \bar{f_{2}}$, $f_{1}\in C_{\cusp}(G(\R),\zeta), f_{2}\in C(G(\R),\zeta)$, $G$ is a reductive $K$-group. The key point in the stable trace formula is that the stable distribution $S^{G^{\prime}}_{M^{\prime}}(\delta,f)$ only depends on the quasisplit group $G^{\prime}$. We have the following theorem.

\begin{theorem} \label{theorem,ls}
If $G$ is a quasisplit $K$-group, we have
\[S^{G}(f)=S^{G}_{\disc}(f).\] \text Here
\begin{equation} \label{eq:lsfg}
S^{G}(f)=\sum_{M\in\mathcal{L}}|W^{M}_{0}||W^{G}_{0}|^{-1}(-1)^{\dimm(A_{M}/A_{G})}\int_{\Delta_{G-\reg,\elll}(M,\zeta)}n(\delta)^{-1}S^{G}_{M}(\delta,f_{1})\overline{f^{G}_{2}(\delta)}\,d\delta,
\end{equation}
and
\begin{equation} \label{eq:lsfp}
S^{G}_{\disc}(f)=\int_{\Phi_{2}(G,\zeta)}S^{G}(\phi)\widetilde{f}^{G}_{1}(\phi)\overline{f^{G}_{2}(\phi)}\,d\phi.
\end{equation}
\end{theorem}
\begin{proof}
 Because we have \[ S^{G}(f)=I(f)-\sum_{G^{\prime}\in\mathcal{E}^{\circ}_{\elll}(G)}\iota(G,G^{\prime})\widehat{S}^{G^{\prime}}(f^{\prime}),\]
and \[ S^{G}_{\disc}(f)=I_{\disc}(f)-\sum_{G^{\prime}\in\mathcal{E}^{\circ}_{\elll}(G)}\iota(G,G^{\prime})\widehat{S}^{G^{\prime}}_{\disc}(f^{\prime}),\]
and $f_{1}$ is cuspidal, then $I(f)=I_{\disc}(f)$. We can prove the theorem inductively. Since $G$ is quasisplit, we have $\dimm(G^{\prime})< \dimm(G)$ for all $G^{\prime}\in\mathcal{E}^{\circ}_{\elll}(G)=\mathcal{E}_{\elll}(G)\backslash\{G\}$. So for all $G^{\prime}\in \mathcal{E}^{\circ}_{\elll}(G)$, we have $\widehat{S}^{G^{\prime}}(f^{\prime})=\widehat{S}^{G^{\prime}}_{\disc}(f^{\prime})$ by the induction, thus $S^{G}(f)=S^{G}_{\disc}(f)$.
\end{proof}

\bigskip
We need to connect the distributions of the geometric side and the distributions of the spectral side in the stable local trace formula. To do this, we need the stable Weyl integral formula. We recall the Weyl integral formula, which is an expansion
\begin{equation}\label{Weyl integral}
\Theta(\pi,f_{2})=\sum_{M\in\mathcal{L}}|W^{M}_{0}||W^{G}_{0}|^{-1}\int_{\Gamma_{\elll}(M(\R),\zeta)}\Phi_{M}(\pi,\gamma)I_{G}(\gamma,f_{2})\,d\gamma, \end{equation}
 where $\Phi_{M}(\pi,\gamma)=|D(\gamma)|^{1/2}\Theta(\pi,\gamma)$, and $\Theta(\pi,\gamma)=\tr\pi(\gamma).$

\bigskip
We know that
\[ f^{G}_{2}(\phi)=\sum_{\pi\in\Pi_{\phi}}\Theta(\pi,f_{2}).\]
Substitute the Weyl integral formula \ref{Weyl integral} into this formula, we obtain an expansion, \[f^{G}_{2}(\phi)=\sum_{\pi\in\Pi_{\phi}}\sum_{M\in\mathcal{L}}|W^{M}_{0}||W^{G}_{0}|^{-1}\int_{\Gamma_{\elll}(M(\R),\zeta)}\Phi_{M}(\pi,\gamma)I_{G}(\gamma,f_{2})\, d\gamma.\]

We need to recall the basic objects of geometric side before stabilizing the Weyl integral formula. We shall make free use of the language and notation  \cite{LS} in this part, often without comments. To define the general transfer factor $\Delta(\sigma^{\prime},\gamma),$  it is necessary to fix elements $\bar{\sigma^{\prime}}$ and $\bar{\gamma}$ such that $\bar{\sigma^{\prime}}$ is an image of $\bar{\gamma}$, and to specify $\Delta(\bar{ \sigma ^{\prime}},\bar{\gamma})$. We will take it to be any complex number of absolute value 1, then $\Delta(\sigma^{\prime},\gamma)$ is defined to be the product of $\Delta(\bar{ \sigma ^{\prime}},\bar{\gamma})$  with the factor \[ \Delta(\sigma^{\prime},\gamma;\bar{ \sigma ^{\prime}},\bar{\gamma})=\frac{\Delta_{I}(\sigma^{\prime},\gamma)}{\Delta_{I}(\bar{ \sigma ^{\prime}},\bar{\gamma})}\frac{\Delta_{II}(\sigma^{\prime},\gamma)}{\Delta_{II}(\bar{ \sigma ^{\prime}},\bar{\gamma})}\frac{\Delta_{2}(\sigma^{\prime},\gamma)}{\Delta_{2}(\bar{ \sigma ^{\prime}},\bar{\gamma})}\Delta_{1}(\sigma^{\prime},\gamma;\bar{ \sigma ^{\prime}},\bar{\gamma}).
\]
There is an additional factor $\Delta_{IV}(\sigma^{\prime},\gamma)=|D^{G}(\gamma)||D^{G^{\prime}}(\gamma^{\prime})|^{-1}$ included in the definition of \cite{LS}, but since we have already put these normalizing factors into our orbital integrals, the term does not appear in this equation. The remaining factors are all constructed from the special values of unitary abelian characters, and therefore have absolute value 1.

\bigskip
There is a natural measure on $\Gamma_{\elll}(G,\zeta)$ given by
\[ \int_{\Gamma_{\elll}(G,\zeta)}\alpha(\gamma)\,d\gamma=\sum_{\{T\}}|W(G(\R),T(\R))|^{-1}\int_{T(\R)/Z(\R)}\alpha(t)\,dt
\]
for any $\alpha\in C(\Gamma_{\elll}(G,\zeta))$, where $\{T\}$ is a set of representatives of $G(\R)$ conjugacy classes of elliptic torus in $G$ over $\R$, $W(G(\R),T(\R))$ is the Weyl group of $(G(\R),T(\R))$, and $dt$ is the Haar measure on $T(\R)$. $\Gamma_{\elll}(G,\zeta)$ is the set of conjugacy classes $\gamma$ in $G(\R)$ such that $G_{\gamma}$ is an elliptic maximal torus in $G$ and as a distribution with a central character $\zeta$ on $Z(\R)$ as in \cite{A9}.
Let $\Delta_{\elll}(G,\zeta)$ be the set of stable conjugacy classes in $\Gamma_{\elll}(G,\zeta)$. We define a measure on $\Delta_{\elll}(G,\zeta)$ by setting
 \[ \int_{\Delta_{\elll}(G,\zeta)}\beta(\delta)\,d\delta=\sum_{\{T\}_{\stab}}|W_{\R}(G,T)|^{-1}\int_{T(\R)/Z(\R)}\beta(t)\,dt\]
for any $\beta\in C(\Delta_{\elll}(G,\zeta))$, where $\{T\}_{\stab}$ is a set of representatives of stable conjugacy classes of elliptic maximal tori in $G$ over $\R$. And $W_{\R}(G,T)$ is the subgroup of elements in the absolute Weyl group of $(G,T)$ defined over $\R$. The measure on $\Gamma_{\elll}(G,\zeta)$ and $\Delta_{\elll}(G,\zeta)$ are related by a formula
\begin{equation}\label{eq:gisl}
 \int_{\Delta_{\elll}(G,\zeta)}(\sum_{\gamma\to\delta}\alpha(\gamma))\,d\delta=\int_{\Gamma_{\elll}(G)}\alpha(\gamma)\,d\gamma.
\end{equation}

\bigskip
Let $\Gamma^{\mathcal{E}}_{\elll}(G,\zeta)$  be the set of isomorphism classes of pair $(G^{\prime},\sigma^{\prime})$, where $G^{\prime}$ is an elliptic endoscopic datum for $G$ and $\sigma^{\prime}$ in an element in $\Delta_{G,\elll}(G^{\prime},\zeta)$. By an isomorphism from $(G^{\prime},\sigma^{\prime})$ to a second pair $(G^{\prime}_{1},\sigma^{\prime}_{1})$, we mean that an isomorphism from the datum $G^{\prime}$ to $G^{\prime}_{1}$, which takes $\sigma^{\prime}$ to $\sigma^{\prime}_{1}$. So we have a decomposition
\[\Gamma^{\mathcal{E}}_{\elll}(G,\zeta)=\coprod_{G^{\prime}\in\mathcal{E}_{\elll}(G)}\Delta_{\elll}(G^{\prime},G,\zeta),\]
where  $\Delta_{\elll}(G^{\prime},G,\zeta)=\Delta_{G,\elll}(G^{\prime},\zeta)/\Out_{G}(G^{\prime})$. We also have an analogue of Lemma \ref{lemma:cmm} for the geometric transfer factor.

\begin{lemma} \label{lemma:gmmc}
Suppose that $\alpha\in C(\Gamma_{\elll}(G,\zeta))$, and that $ \beta\in C(\Gamma^{\mathcal{E}}_{\elll}(G,\zeta))$. Then
\[\int_{\Gamma_{\elll}(G,\zeta)}\sum_{\delta\in\Gamma^{\mathcal{E}}_{\elll}(G,\zeta)}\beta(\delta)\Delta(\gamma,\delta)\alpha(\gamma)\,d\gamma=\int_{\Gamma^{\mathcal{E}}_{\elll}(G,\zeta)}\sum_{\gamma\in\Gamma_{\elll}(G,\zeta)}\beta(\delta)\Delta(\gamma,\delta)\alpha(\gamma)\,d\delta.
,\]
\end{lemma}
\begin{proof}
 Let $\psi:G\to G^{\ast}$ be the underlying quasisplit inner twist of $G$. According to
\eqref{eq:gisl} the integral over $\Gamma_{\elll}(G,\zeta)$ can be decomposed into an integral over $\delta^{\ast}\in\Delta_{\elll}(G^{\ast},\zeta)$ and a sum over the elements $\gamma\in\Gamma_{\elll}(G,\zeta)$ which is mapped to $\delta^{\ast}$. Similarly, the integral over $\Gamma^{\mathcal{E}}_{\elll}(G,\zeta)$ can be decomposed into an integral over $\delta^{\ast}\in\Delta_{\elll}(G^{\ast},\zeta)$, and a summation over the element $\delta^{\prime}\in\Gamma^{\mathcal{E}}_{\elll}(G,\zeta)$ which is mapped to $\delta^{\ast}$. This depends on the fact that the map $\delta^{\prime} \to S_{T^{\ast}}(\delta^{\prime})$, $T^{\ast}=G_{\delta^{\ast}}$ is a bijection from the preimage of $\delta^{\ast}$ in $\Gamma^{\mathcal{E}}_{\elll}(G,\zeta)$ onto $\mathcal{K}(T^{\ast})=\pi_{0}((\widehat{T^{\ast}})^{\Gamma}/Z(\widehat{G})^{\Gamma})$, and this bijection depends on $G$, being a $K$-group.

  With the two decomposition, we can represent each side of the required identity as an integral over $\Delta_{\elll}(G^{\ast},\zeta),$ and a double sum over $\delta$ and ${\gamma}$. The transfer factor $\Delta(\delta,\gamma)$ vanishes unless $\delta$ and $\gamma$ have the same image in the $\Delta_{\elll}(G^{\ast},\zeta)$. The double sum in each case can therefore be taken over the preimages of $\delta^{\ast}$ in $\Gamma^{\mathcal{E}}_{\elll}(G,\zeta)\times \Gamma_{\elll}(G,\zeta)$, then the identity follows.
  \end{proof}

\bigskip
 We can now stabilize the Weyl integral formula.
  \begin{lemma}
  For $f\in C(G(\R),\zeta).$
  Then we have the stable Weyl integral formula  \[f^{G}_{2}(\phi)=\sum_{M\in\mathcal{L}}|W^{M}_{0}||W^{G}_{0}|^{-1}\int_{\Gamma^{\mathcal{E}}_{\elll}(M(\R),\zeta)}n(\delta)^{-1}\sum_{\pi\in\Pi_{\phi}}S\Phi_{M}(\pi,\delta)f^{\mathcal{E}}_{2,M}(\delta)\,d\delta,
  \]
  where $S\Phi_{M}(\pi,\delta)=\sum_{\gamma\in\Gamma_{\elll}(M(\R),\zeta)}\overline{\Delta(\delta,\gamma)}\Phi_{M}(\pi,\gamma)$.
  \end{lemma}
  \begin{proof}
 We observe that $I_{G}(\gamma,f_{2})=I^{M}_{M}(\gamma,f_{2})=f_{2,M}(\gamma)$ where $\gamma\in\Gamma_{\elll}(M(\R),\zeta)$,
 and \begin{equation}
 \Theta(\pi,f_{2})=\sum_{M\in\mathcal{L}}|W^{M}_{0}||W^{G}_{0}|^{-1}\int_{\Gamma_{\elll}(M(\R),\zeta)}n(\delta)^{-1}\Phi_{M}(\pi,\gamma)\sum_{\delta\in\Gamma^{\mathcal{E}}_{\elll}(M(\R),\zeta)}\overline{\Delta(\delta,\gamma)}f^{\mathcal{E}}_{2,M}(\delta)\,d\gamma.
\end{equation}
Applying Lemma \ref{lemma:gmmc}, we have \[ \Theta(\pi,f_{2})= \sum_{M\in\mathcal{L}}|W^{M}_{0}||W^{G}_{0}|^{-1}\int_{\Gamma^{\mathcal{E}}_{\elll}(M(\R),\zeta)}n(\delta)^{-1}S\Phi_{M}(\pi,\delta)f^{\mathcal{E}}_{2,M}(\delta)\,d\delta,
\]
where $S\Phi_{M}(\pi,\delta)=\sum_{\gamma\in\Gamma_{\elll}(M(\R),\zeta)}\overline{\Delta(\delta,\gamma)}\Phi_{M}(\pi,\gamma)$.
 So we obtain \[f^{G}_{2}(\phi)=\sum_{M\in\mathcal{L}}|W^{M}_{0}||W^{G}_{0}|^{-1}\int_{\Gamma^{\mathcal{E}}_{\elll}(M(\R),\zeta)}n(\delta)^{-1}\sum_{\pi\in\Pi_{\phi}}S\Phi_{M}(\pi,\delta)f^{\mathcal{E}}_{2,M}(\delta)\,d\delta.\]
\end{proof}

We now obtain the formula for the stable distribution $S^{G}_{M}(\delta,f_{1})$ by comparing the stable Weyl integral formula with the stable local trace formula.

\begin{theorem} \label{theorem:smt}
Assume that $G$ is a quasisplit $K$-group, and $f_{1}\in C_{\cusp}(G(\R),\zeta)$. Then we have
 \begin{equation} \label{eq:stfm} S^{G}_{M}(\delta,f_{1})=(-1)^{\dimm(A_{M}/A_{G})}\int_{\Phi_{2}(G,\zeta)}S^{G}(\phi)S\Phi_{M}(\phi,\delta)\widetilde{f}_{1}^{G}(\phi)\,d\phi,
\end{equation}
where \[S\Phi_{M}(\phi,\delta)=\sum_{\pi\in\Pi_{\phi}}\sum_{\gamma\in\Gamma_{\elll}(M,\zeta)}\overline{\Phi_{M}(\pi,\gamma)}\Delta(\delta,\gamma),\]
and
\[\Phi_{M}(\pi,\gamma)=\begin{cases}
|D(\gamma)|^{1/2}\Theta(\pi,\gamma)& \text{if $\gamma\in M(\R)_{\elll}$ },\\
0& \text{ otherwise},
\end{cases}\]
\[ \widetilde{f}_{1}^{G}(\phi)=\sum_{\tau\in T_{\elll}(G,\zeta)}\Delta(\phi,\tau)|d(\tau)|^{-1}\Theta(\tau,f_{1}),\]
for Levi subgroup $M\in\mathcal{L}$ and $\delta$ is the stable strongly $G$-regular conjugacy class in $M(\R)$.
\end{theorem}
\begin{proof}
 Suppose that $\delta$ does not lie in $\Delta_{\elll}(M(\R),\zeta)$. Then by descent formula \cite[\S 6]{A8} and the cuspidality of $f_{1}$,  $S^{G}_{M}(\delta,f_{1})$ vanishes. The right hand side of \eqref{eq:stfm} vanishes by definition. So the formula holds in the case. It is therefore enough to establish \eqref{eq:stfm} when $\delta$ lies in $\Delta_{\elll}(M(\R),\zeta)$.

  To deal with the elliptic point in $M(\R)$, we apply the simple version of the stable local trace formula. Consider the two expression
 \eqref{eq:lsfg} and \eqref{eq:lsfp} in Theorem \ref{theorem,ls}, with $f_{1}$ for the given cuspidal function and $f_{2}$ for a variable function in $C(G(\R),\zeta)$. The expressions depend on $f_{2}$ through different distributions $f^{G}_{2}$ and $f^{\mathcal{E}}_{2,M}$. However, the relation is given by a stable Weyl integral formula, which has an expansion
\[ f^{G}_{2}(\phi)=\sum_{M\in\mathcal{L}}|W^{M}_{0}||W^{G}_{0}|^{-1}\int_{\Gamma^{\mathcal{E}}_{\elll}(M(\R),\zeta)}n(\delta)^{-1}\sum_{\pi\in\Pi_{\phi}}S\Phi_{M}(\pi,\delta)f^{\mathcal{E}}_{2,M}(\delta)\,d\delta.
\]

We obtain \[ f^{G}_{2}(\phi)=\sum_{M\in\mathcal{L}}|W^{M}_{0}||W^{G}_{0}|^{-1}\int_{\Gamma^{\mathcal{E}}_{\elll}(M(\R),\zeta)}n(\delta)^{-1}S\Phi_{M}(\phi,\delta)f^{\mathcal{E}}_{2,M}(\delta)\,d\delta.
\]
Substituting this into \eqref{eq:lsfp}, we collect the coefficient of $\overline{f^{M}_{2}(\delta)}$ in the resulting identity of \eqref{eq:lsfg} with \eqref{eq:lsfp}. We see that if we set:
\[ P_{M}(\delta,f_{1})=S^{G}_{M}(\delta,f_{1})-(-1)^{\dimm(A_{M}/A_{G})}\int_{\Phi_{2}(G,\zeta)}S^{G}(\phi)S\Phi_{M}(\phi,\delta)\widetilde{f}_{1}^{G}(\phi)\,d\phi,
\]
 then the sum of
  \begin{equation} \label{eq:csl} \sum_{M\in\mathcal{L}}|W^{M}_{0}||W^{G}_{0}|^{-1}(-1)^{\dimm(A_{M}/A_{G})}\int_{\Delta_{G,\elll}(M,\zeta)}n(\delta)^{-1}P_{M}(\delta,f_{1})\overline{f^{M}_{2}(\delta)}\,d\delta,
 \end{equation}
 and \begin{equation}\label{eq:csll}
 -\sum_{M\in\mathcal{L}}|W^{M}_{0}||W^{G}_{0}|^{-1}\sum_{M^{\prime}\in\mathcal{E}^{0}_{\elll}(M)}\int_{\Delta_{\elll}(M^{\prime},M,\zeta)}\int_{\Phi_{2}(G,\zeta)}n(\delta)^{-1}S^{G}(\phi)\widetilde{f}^{G}_{1}(\phi)S\Phi_{M}(\phi,\delta)\overline{f^{M^{\prime}}_{2}(\delta)}\,d\delta d\phi
 \end{equation}
vanishes.

 We have written $f^{\mathcal{E}}_{2,M}(\delta^{\prime})=f^{M^{\prime}}_{2}(\delta^{\prime})$, where $\delta^{\prime}\in\Delta_{G,\elll}(M^{\prime},\zeta)$.
As in \cite[\S 10]{A8}, we choose $f_{2}$ so that $f^{\mathcal{E}}_{2,G}$ has compact support modulo $Z(\R)$ on $\Gamma^{\mathcal{E}}(G)=\coprod_{\{M\}}\Gamma^{\mathcal{E}}_{\elll}(M(\R),\zeta)$ , and so that $f^{\mathcal{E}}_{2,G}$ approaches the $\zeta^{-1}$-equivariant Dirac measure at the image of $\delta^{\prime}$ in $\Gamma^{\mathcal{E}}(G)$. The expression \label{eq:csl} then approaches a nonzero multiple of $P_{M}(\delta^{\prime},f_{1})$ when $\delta^{\prime}\in\Delta_{G,\elll}(M,\zeta)$, and \label{eq:csll} approaches a zero.
We conclude that $P_{M}(\delta,f_{1})=0$. So we have obtained the required formula
\[S^{G}_{M}(\delta,f_{1})=(-1)^{\dimm(A_{M}/A_{G})}\int_{\Phi_{2}(G,\zeta)}S^{G}(\phi)S\Phi_{M}(\phi,\delta)\widetilde{f_{1}^{G}}(\phi)\,d\phi.\]
\end{proof}

\begin{corollary}
 If we set $S^{G}_{M}(M^{\prime},\delta,f)=n(\delta)^{-1}\int_{\Phi_{2}(G,\zeta)}S^{G}(\phi)S\Phi_{M}(\phi,\delta)\widetilde{f}^{G}(\phi)\mathrm{d}\phi$, and $f$ is a cuspidal function, where $\delta\in \Delta_{ell}(M^{\prime},M,\zeta)$. Then $S^{G}_{M}(M^{\prime},\delta,f)$ vanishes.
\end{corollary}
  The proof of the Corollary comes from the process of the proof of Theorem \ref{theorem:smt}.

\bigskip
\section{Multiplicity formula of discrete series} \label{sec:mf}

We now return to the discussion of section \ref{sec:mi}. In order to establish the multiplicity formula of discrete series, we descend from the $K$-group to the connected reductive group. If we take the test function whose components vanish except for the one from the required connected group, then we can apply the properties of a $K$-group to connected reductive group. We assume that the infinitesimal character $\mu$ is regular. We take the center $Z(\R)=Z(G(\R))$, then $\mathfrak{a}^{\ast}_{G,Z}=1$, and the formula for the stable distribution $S^{G}_{M}(\delta,f_{1})$ simplifies. $\Pi_{\phi}$ is in bijection with $\mathcal{E}(T)$, where $T$ is a maximal torus of $G(\R)$ that is compact modulo centre, and $R_{\phi}$ is trivial. So $|i^{G}(\tau)|$ is trivial, and $\tau=\pi_{\R}\in\Pi_{2}(G(\R),\zeta)$, then
\begin{equation}
\begin{split}
\widetilde{f}_{1}^{G^{\prime}}(\phi^{\prime})&=\sum_{\tau\in T_{\elll}(G,\zeta)}\Delta(\phi^{\prime},\tau)|i^{G}(\tau)|\Theta(\tau,f_{1}) \\
&=\sum_{\pi\in\Pi_{2}(G(\R),\zeta)}\Delta(\phi^{\prime},\pi)\tr\pi(f_{1}) \\
&=f^{G^{\prime}}_{1}(\phi^{\prime}), \\
\end{split}
\end{equation}
and $f^{G^{\prime}}_{1}(\phi^{\prime})=\sum_{\pi\in\Pi_{\phi^{\prime}}}\Theta(\pi,f_{1})$.

\bigskip
 Now our main obstruction is that the pseudo-coefficient $f_{\pi_{\R}}$ is cuspidal, but not stable. However $f_{\pi_{\R}}$ transfers to $f^{\prime}_{\phi^{\prime}_{\mu}}$ on $G^{\prime}$ by Shelstad transfer theorem, where $\phi^{\prime}_{\mu}$ can be parameterized by $\mu^{\prime}$, such that $\mu=\mu^{\prime}+\mu^{\ast}$, and $\mu^{\ast}$ parameterizes the embedding $\xi^{\prime}$, which is a part of the endoscopic data of $G^{\prime}$. Then $f^{\prime}_{\phi^{\prime}_{\mu}}(\phi^{\prime})$
is stable cuspidal, where \[f^{\prime}_{\phi^{\prime}_{\mu}}(\phi^{\prime})=\sum_{\pi\in\Pi_{\phi^{\prime}}}\tr\pi(f^{\prime}_{\phi^{\prime}_{\mu}})\] \[=\sum_{\pi\in\Pi_{2}(G(\R),\zeta)}\Delta(\phi^{\prime},\pi)tr\pi(f_{\pi_{\R}})=\Delta(\phi^{\prime},\pi_{\R}).\]

Then
\[f^{\prime}_{\phi^{\prime}_{\mu}}(\phi^{\prime})=
\begin{cases}
\Delta(\phi^{\prime}_{\mu},\pi_{\R})& \text{if $\phi^{\prime}=\phi^{\prime}_{\mu}$},\\
0& \text{otherwise}.
\end{cases}\]

We take $f^{\prime}_{\phi^{\prime}_{\mu}}=\Delta(\phi^{\prime}_{\mu},\pi_{\R})f_{\phi^{\prime}_{\mu}}$ and $f_{\phi^{\prime}_{\mu}}=\frac{1}{|\mathcal{S}_{\phi^{\prime}_{\mu}}|}\sum_{\pi\in\Pi_{\phi^{\prime}_{\mu}}}f_{\pi}$, where $f_{\pi}$ is pseudo-coefficient of $\pi$.

\bigskip
  We assume that $G$ is a quasisplit connected reductive group, the test function $f$ is stable cuspidal on $G$. Then we get a simple stable distribution from \eqref{eq:stfm}, \[S^{G}_{M}(\delta,f)=(-1)^{\dimm(A_{M}/A_{G})}\sum_{\phi\in\Phi_{2}(G,\zeta)}S^{G}(\phi)S\Phi_{M}(\phi,\delta)f^{G}(\phi),\]
where
\[S\Phi_{M}(\phi,\delta)=\sum_{\pi\in\Pi_{\phi}}\sum_{\gamma\in\Gamma_{\elll}(M,\zeta)}\Delta(\delta,\gamma)\overline{\Phi_{M}(\pi,\gamma)},\] \[\Phi_{M}(\pi,\gamma)=|D^{G}(\gamma)|^{1/2}\Theta_{\pi}(\gamma)=I_{M}(\pi,\gamma)\] as in \cite{A3}, and \[f^{G}(\phi)=\sum_{\pi\in\Pi_{\phi}}\tr\pi(f)=|\mathcal{S}_{\phi}|\overline{tr\widetilde{\pi}(\bar{f})}.\]

So \[ S^{G}_{M}(\delta,f)=(-1)^{\dimm(A_{M}/A_{G})}\sum_{\phi\in\Phi_{2}(G,\zeta)}\sum_{\pi\in\Pi_{\phi}}\sum_{\gamma\in\Gamma_{\elll}(M,\zeta)}|\mathcal{S}_{\phi}|S^{G}(\phi)\Delta(\delta,\gamma)\overline{I_{M}(\pi,\gamma)}\overline{\tr\widetilde{\pi}(\bar{f})},
\]
where \[|\mathcal{S}_{\phi}|S^{G}(\phi)=|\mathcal{S}_{\phi}||\mathcal{S}_{\phi}|^{-1}=1.\]
However,
\[
\begin{split}
\sum_{\phi\in\Phi_{2}(G,\zeta)}\sum_{\pi\in\Pi_{\phi}}\overline{I_{M}(\pi,\gamma)tr\widetilde{\pi}(\bar{f})}&=\sum_{\pi\in\Pi_{2}(G(\R),\zeta)}\overline{I_{M}(\pi,\gamma)\tr\widetilde{\pi}(\bar{f})} \\
&=(-1)^{\dimm(A_{M}/A_{G})}\vol(T(\R)/A_{M}(\R)^{0})\overline{I_{M}(\gamma,\bar{f})} \\
&=(-1)^{\dimm(A_{M}/A_{G})}|D^{M}(\gamma)|^{1/2}\vol(T(\R)/A_{M}(\R)^{0})\overline{\Phi_{M}(\gamma,\bar{f})} .\\
\end{split}
\]

Then we obtain the following theorem.

\begin{theorem}
Suppose that $G$ is a quasisplit $K$-group, $f\in C(G(\R),\zeta)$ is stable cuspidal and $\delta\in \Delta(M(\R),\zeta)$,
then \[
S^{G}_{M}(\delta,f)=(-1)^{\dimm(A_{M}/A_{G})}\sum_{\phi\in\Phi_{2}(G,\zeta)}S^{G}(\phi)S\Phi_{M}(\phi,\delta)f^{G}(\phi)\]
\begin{equation}\label{formula}
=\sum_{\gamma\in\Gamma_{\elll}(M,\zeta)}\Delta(\delta,\gamma)|D^{M}(\gamma)|^{1/2}\vol(T(\R)/A_{M}(\R)^{0})\overline{\Phi_{M}(\gamma,\bar{f})}.
\end{equation}
\end{theorem}
In particular, $S^{G}_{M}(\delta,f)$ vanishes, if $\delta$ is not semisimple.
\begin{proof}
 We just need to check that the stable distribution $S^{G}_{M}(\delta,f)$ vanishes, if $\delta$ is not semi-simple. We know $\Phi_{M}(\gamma,\bar{f})$ vanishes, if $\gamma$ is not semisimple, and $f\in\mathcal{H}_{\ac}(G(\R),\zeta)$ is stable cuspidal. However,  $\mathcal{H}_{\ac}(G(\R),\zeta)$ is dense in $C(G(\R),\zeta)$, so if $f\in C(G(\R),\zeta)$ is stable cuspidal, then $f$ can be approached by $f^{\prime}\in \mathcal{H}_{\ac}(G(\R),\zeta)$, which is also stable cuspidal. we can use the trace Paley-Wiener theorem to extend the result to $C(G(\R),\zeta)$.
\end{proof}
We denote $S\Phi^{G}_{M}(\delta,f)=|D^{M}(\delta)|^{-\frac{1}{2}}S^{G}_{M}(\delta,f)$. Form the above theorem, we have defined $S\Phi^{G}_{M}(\delta,f)$ on the stable conjugacy classes of strongly regular points in $M(\mathbb{R})$. $\Delta(\delta,\gamma)$ is a continuous function on $\delta$ and the point $\gamma\in M(\mathbb{R})$. The summation in (\ref{formula}) is finite sum by the property of transfer factor. If the transfer factor $\Delta(\delta,\gamma)\neq 0$ for given $\delta$, then $|D^{M}(\delta)|=|D^{M}(\gamma)|$. So $S\Phi_{M}^{G}(\delta,f)$ is a continuous function on the stable conjugacy classes of strongly regular points in $M(\R)$.

\bigskip
We can now give a dimension formula for spaces of automorphic forms. For each $\pi_{\R}\in\Pi_{2}(G(\R),\zeta)$, let $m_{\disc}(\pi_{\R},K_{0})$ be the multiplicity of $\pi_{\R}$ in
\begin{equation}\label{eq:spdc}
 L^{2}(G(\Q)\backslash G(\A)/K_{0},\zeta)=\oplus^{n}_{i=1} L^{2}(\Gamma_{i}\backslash G(\R),\zeta),
\end{equation}
where $K_{0}$ is an open compact subgroup of $G(\A_{\fin})$, and $\{\Gamma_{i}\}$  are the discrete subgroup. Let $h$ be a $K_{0}$ bi-invariant function  in $\mathcal{H}(G(\A_{\fin}))$. Let $R_{\disc}(\pi_{\R},h)$ be the operator on $\pi_{\R}$-isotypical subspace of \eqref{eq:spdc}, it can be interpreted as a $\bigg(m_{\disc}(\pi_{\R},K_{0})\times m_{\disc}(\pi_{\R},K_{0})\bigg)$ matrix.

Then Proposition \ref{prop:cmf} yields the formula,
\begin{equation*}
\begin{split}
 &\tr(R_{\disc}(\pi_{\R},h))=I(f_{\pi_{\R}}h)\\
 =\sum_{G^{\prime}\in\mathcal{E}_{\elll}(G)}\iota(G,G^{\prime})&\sum_{M^{\prime}\in\mathcal{L}^{G^{\prime}}}|W^{M^{\prime}}_{0}||W^{G^{\prime}}_{0}|^{-1}\sum_{\delta\in\{M^{\prime}(\Q)\}_{M^{\prime},S}}b^{M^{\prime}}(\delta)S^{G^{\prime}}_{M^{\prime}}(\delta,f_{\pi_{\R}})(h_{M})^{M^{\prime}}(\delta),
 \end{split}
 \end{equation*}
where $M\in\mathcal{L}(G)$ and $M^{\prime}\in\mathcal{E}(M)$. The function $f_{\pi_{\R}}$ is a pseudo-coefficient of $\pi_{\R}$, which belongs to $C_{\cusp}(G(\R),\zeta)$, and
\[S^{G^{\prime}}_{M^{\prime}}(\delta,f_{\pi_{\R}})=\widehat{S}^{G^{\prime}}_{M^{\prime}}(\delta,f^{\prime}_{\phi^{\prime}_{\R}})=(-1)^{\dimm(A_{M^{\prime}}/A_{G^{\prime}})}S^{G^{\prime}}(\phi^{\prime}_{\mu})S\Phi_{M^{\prime}}(\phi^{\prime}_{\mu},\delta)\Delta(\phi^{\prime}_{\mu},\pi_{\R}).\]

In particular, the stable distribution vanishes unless $\delta$ is semisimple. The set of equivalence classes in $\{M(\Q)\}_{M,S}$ is just $M(\Q)$-semisimple stable conjugacy classes. Moreover, for any semisimple elliptic stable conjugacy classes $\delta\in \{M(\Q)\}_{\elll},$
 we have the global coefficient $b^{M^{\prime}}(\delta)=b^{M^{\prime}}_{\elll}(\delta)=\tau(M^{\prime})$.

  We discuss the transfer of $h_{M}(\gamma)$. For the given function $h\in\mathcal{H}(G(\mathbb{A}_{\fin}))$, we can find a function $h^{G^{\prime}}\in\mathcal{H}(G^{\prime}(\mathbb{A}_{\fin}))$ whose orbital integrals match those of $h$. In other words we need $h^{G^{\prime}}\in\mathcal{H}(G^{\prime}(\mathbb{A}_{\fin}))$ such that for all $\delta\in G^{\prime}_{\reg}(\mathbb{A}_{\fin})$
  \[ (h_{M})^{M^{\prime}}(\delta)=\sum_{\gamma}\Delta(\delta,\gamma)h_{M}(\gamma),\]
where the sum is taken over $M(\mathbb{A}_{\fin})$-conjugacy classes of images $\gamma\in M(\mathbb{A})_{\fin}$ of $\delta\in M^{\prime}(\mathbb{A})_{\fin}$. It follows from Ngo's proof of the Fundamental Lemma, and Waldspurger's results \cite{W1},\cite{W2} that the Fundamental Lemma implies the transfer. Thus the function $(h_{M})^{M^{\prime}}$ exists, and $h_{M}$ is defined in terms of $h$ by the formula:
\[ h_{M}(\gamma)=\delta_{P}(\gamma_{\fin})^{1/2}\int_{K_{\fin}}\int_{N_{P}(\A_{\fin})}\int_{M_{\gamma}(\A_{\fin})\backslash M(\A_{\fin})} h(k^{-1}m^{-1}\gamma mnk)\,dmdndk.
\]
Here $P=MN_{P}$ is any parabolic subgroup with Levi component $M$, $\delta_{P}(\gamma_{\fin})$ is the modular function of $P$, evaluated at the image of $\gamma$ in $G(\A_{\fin})$. In particular, the stable orbital integral on $M^{\prime}$ can be taken over $M^{\prime}(\A_{\fin})$ rather then $M^{\prime}(\Q_{S_{0}})$. Here $S_{0}=S-\{\infty\}$, and we have no further need to single out the finite set of valuations.
\begin{remark}
 I thank Waldspurger for pointing out the following fact. The Fundamental Lemma is proved in the unramified case. If the situation is not unramified, the problem is more complicated. There are several conjugacy classes of maximal compact subgroups, and there is no natural correspondence between maximal compact subgroups of $G$ and maximal subgroups of $G^{\prime}$. So we cannot expect a simple formula for the transfer of the characteristic function of some maximal compact subgroup of $G$. Maybe, we can hope that this transfer is a linear combination of characteristic functions of maximal compact subgroups of $G^{\prime}$.
\end{remark}
We set
\[ P_{\mu}(M^{\prime})=S^{G^{\prime}}(\phi^{\prime}_{\mu})\Delta(\phi^{\prime}_{\mu},\pi_{\R})\tau(M^{\prime})\]
and we also write $\{M(\Q)\}$ for the set of $M(\Q)$-semisimple stable conjugacy classes in $M(\Q)$. We have obtained the main theorem.
\begin{theorem} \label{theorem: m}
 If $h\in \mathcal{H}(G(\A_{\fin}))$, and the highest weight $\mu$ of representation is regular, then we have
  \begin{equation*}
  \begin{split}
  &\tr(R_{\disc}(\pi_{\R},h))   \\
=&\sum_{G^{\prime}\in\mathcal{E}_{\elll}(G)}\iota(G,G^{\prime})\sum_{M^{\prime}\in\mathcal{L}^{G^{\prime}}}(-1)^{\dimm(A_{M^{\prime}}/A_{G^{\prime}})}|W^{M^{\prime}}_{0}||W^{G^{\prime}}_{0}|^{-1}\sum_{\delta\in\{M^{\prime}(\Q)\}}P_{\mu}(M^{\prime})S\Phi_{M^{\prime}}(\phi^{\prime}_{\mu},\delta)(h_{M})^{M^{\prime}}(\delta),
\end{split}
\end{equation*}
 and the multiplicity formula of the discrete series \[  m_{\disc}(\pi_{\R},K_{0})=\tr(R_{\disc}(\pi_{\R},\mathit{1}_{K_{0}})).
   \]
\end{theorem}

\begin{remark}
\begin{enumerate}

\item The sum in $\delta$ can be taken over a finite sum that depends only on the support of $h$, so the theorem therefore provides a finite closed formula for $\tr(R_{\disc}(\pi_{\R},h))$.
\end{enumerate}
\end{remark}

\bigskip
\section{The stable formula of $L^{2}$-Lefschetz number}\label{sec:sl}

 We can give a stable formula for $L^{2}$-Lefschetz number. We can refer to the Arthur's paper \cite{A3}, which gives a formula through the invariant trace formula. We need to recall the basic result in \cite{A3}. We still use the notation in \cite{A3}. If $h\in\mathcal{H}(G(A_{\fin}))$, then the $L^{2}$-Lefschetz number
 \begin{equation}
 \begin{split}
   \mathcal{L}_{\mu}(h)&=\sum_{q}(-1)^{q}\tr(H^{q}_{2}(h,\mathcal{F}_{\mu}))\\
                       &=\sum_{\pi\in\Pi(G(\mathbb{A},\xi))}m_{\disc}(\pi)\chi_{\mu}(\pi_{\mathbb{R}})\tr\pi_{\fin}(h)\\
                       &=\sum_{\pi\in\Pi(G(\mathbb{A},\xi))}m_{\disc}(\pi)\tr\pi_{\mathbb{R}}(f_{\mu})\tr\pi_{\fin}(h) \\
                       &=\sum_{\pi\in\Pi(G(\mathbb{A}),\xi)}m_{\disc}(\pi)\tr(f_{\mu}h)\\
                       &=I(f_{\mu}h),
   \end{split}
 \end{equation}

 where $\chi_{\mu}(\pi_{\mathbb{R}})=\sum_{q}(-1)^{q}\dimm H^{q}(g(\mathbb{R}),K'_{R};\pi_{\mathbb{R}}\otimes\mu))$, and
 \begin{equation}
 \chi_{\mu}(\pi_{\mathbb{R}})=\tr\pi_{\mathbb{R}}(f_{\mu})=
 \begin{cases}
 (-1)^{q(G)}&\quad \text{if $\pi_{\mathbb{R}}\in \Pi_{\disc}(\widetilde{\mu})$},\\
 0          & \quad \text{otherwise}.
 \end{cases}
 \end{equation}
 If the test function $f$ is stable cuspidal, then the invariant distribution is equal to zero on the non-trivial unipotent elements. So we have
 \[I(f_{\mu}h)=\sum_{M\in\mathcal{L}}(-1)^{\dimm(A_{M}/A_{G})}|W^{M}_{0}||W_{0}^{G}|^{-1}\sum_{\gamma\in(M(\mathbb{Q}))}\chi(M_{\gamma})|\iota^{M}(\gamma)|^{-1}\Phi_{M}(\gamma,\mu)h_{M}(\gamma)\]
where $\chi(M_{\gamma})=(-1)^{q(M_{\gamma})}\vol(M_{\gamma}(\mathbb{Q})A_{M_{\gamma}}(\mathbb{R})^{\circ}\backslash M_{\gamma}(\mathbb{A}))\vol(A_{M_{\gamma}}(\mathbb{R})^{\circ}\backslash\overline{M_{\gamma}}(\R))^{-1}|\mathcal{D}(M_{\gamma},B)|$, $\mathcal{D}(G,B)=W(G(\mathbb{R}),B(\mathbb{R}))\backslash W(G,B)$, $|\iota^{M}(\gamma)|=|M_{\gamma}(\mathbb{Q})\backslash M(\mathbb{Q,\gamma})|$, and $(M(\mathbb{Q}))$ is the set of $M(\mathbb{Q})$-conjugacy classes of $M(\mathbb{Q})$.

\bigskip
Since $f_{\mu}$ is stable cuspidal, we can use the Proposition \ref{prop:cmf}. Then we obtain the stable formula
  \[I(f_{\mu}h)=\sum_{G^{'}\in\mathcal{E}_{\elll}(G)}\iota(G,G^{'})\sum_{M^{'}\in\mathcal{L}^{G^{'}}}|W^{M^{'}}_{0}||W^{G^{'}}_{0}|^{-1}\sum_{\delta\in\Delta(M^{'},V,\zeta)}b^{M^{'}}(\delta)S^{G^{'}}_{M^{'}}(\delta,f_{\mu})(h_{M})^{M^{'}}(\delta).\]
Where $S^{G^{\prime}}_{M^{\prime}}(\delta,f_{\mu})=(-1)^{\dimm(A_{M^{\prime}}/A_{G^{\prime}})}\sum_{\phi^{\prime}\in\Phi_{2}(G^{\prime},\zeta)}S^{G^{\prime}}(\phi^{\prime})S\Phi_{M}(\phi^{\prime},\delta)f^{G^{\prime}}_{\mu}(\phi^{\prime})$
  and
  \begin{equation}
f^{G^{\prime}}_{\mu}(\phi^{\prime})=\sum_{\pi\in\Pi_{2}(G(R),\zeta)}\Delta(\phi^{\prime},\pi)\tr\pi(f_{\mu})=
  \begin{cases}
 (-1)^{q(G)}\sum_{\pi\in\Pi_{\phi}}\Delta(\phi^{\prime},\pi) &\quad \text{if $\phi=\phi(\mu,\lambda)$},\\
 0                                                           &\quad  \text{otherwise},
  \end{cases}
  \end{equation}
where $q(G)=\frac{1}{2}\dimm(G(\mathbb{R})/K_{\mathbb{R}})$, and $f^{G^{\prime}}_{\mu}$ is a stable cuspidal function. We obtain the stable distribution
\[S^{G^{\prime}}_{M^{\prime}}(\delta,f_{\mu})=(-1)^{\dimm(A_{M^{\prime}}/A_{G^{\prime}})+q(G)}\sum_{\pi\in\Pi_{\phi(\mu,\lambda)}}\Delta(\phi^{\prime},\pi)S^{G^{\prime}}(\phi^{\prime})S\Phi_{M^{\prime}}(\phi^{\prime},\delta).\]

 \begin{remark}
If $G=G^{\prime}=G^{s=1}$ is a quasisplit group, and $\Delta(\phi^{s},\pi)=\xi(s)$ is just a character \cite{S3}, we have the simple formula for the stable distribution.
  \begin{equation}
 f^{G^{s}}_{\mu}(\phi^{\prime})=
 \begin{cases}
 (-1)^{q(G)}|\mathcal{S}_{\phi}|&\quad \text{if $s=1$ and $\phi=\phi(\mu,\lambda)$},\\
 0                              &\quad \text{otherwise},
 \end{cases}
 \end{equation}
where $|\mathcal{S_{\phi}}|$ is the cardinal number of the $L$-packet of $\phi$.
\end{remark}

Since $f^{G^{\prime}}_{\mu}$ is stable cuspidal, we have $S^{G^{\prime}}_{M^{\prime}}(\delta,f_{\mu})=\widehat{S}^{G^{\prime}}_{M^{\prime}}(\delta, f^{G^{\prime}}_{\mu})=0$, if $\delta$ is a not semisimple element of $M^{\prime}$. While we have $b^{M^{\prime}}(\delta)=\tau(M^{\prime})$, if $\delta$ is a semisimple, elliptic element.
  The set of equivalence classes in $\Delta(M^{\prime},V,\xi)$ equals the set of equivalence classes in $\Delta(\bar{M^{\prime}},V)$, which are just $\bar{M^{\prime}}$-strongly regular stable conjugacy classes, where $\bar{M^{\prime}}=M^{\prime}/Z$.
 We set
  \[F_{\mu}(M^{\prime})=(-1)^{\dimm(A_{M^{\prime}}/A_{G^{\prime}})+q(G^{\prime})}\tau(M^{\prime})S^{G}(\phi^{\prime})\sum_{\pi\in\Pi_{\phi(\mu,\lambda)}}\Delta(\phi^{\prime},\pi),\]
 where $(-1)^{q(G^{\prime})}=(-1)^{q(G)}$, and we also write $\{M(\Q)\}$ for the set of stable $M(\Q)$-semisimple conjugacy classes in $M(\Q)$.
 Thus we obtain the following theorem.

  \begin{theorem}
  For any $h\in \mathcal{H}(G(\mathbb{A}_{\fin}))$, we have
  \[\mathcal{L}_{\mu}(h)=\sum_{G^{'}\in\mathcal{E}_{\elll}(G)}\iota(G,G^{'})\sum_{M^{'}\in\mathcal{L}^{G^{'}}}|W^{M^{'}}_{0}||W^{G^{'}}_{0}|^{-1}\sum_{\delta\in\{M^{\prime}(\mathbb{Q})\}}F_{\mu}(M^{\prime})S\Phi_{M^{\prime}}(\phi^{\prime},\delta)(h_{M})^{M^{\prime}}(\delta).\]

  \end{theorem}


\begin{thebibliography}{}


\bibitem{A1} J. Arthur \emph{The invariant trace formula I. Local theory.} J. Amer. Math. Soc. 1988, 1: 323-383

\bibitem{A2} J. Arthur \emph{The invariant trace formula II. Global theory.} J. Amer. Math. Soc. 1988, 1: 501-554

\bibitem{A3} J. Arthur  \emph{$L^2$-Lefschetz numbers of Hecke operators.} Invent. Math.  1989, 97: 257-290

\bibitem{A4} J. Arthur \emph{A local trace formula.} Pub. Math. I.H.E.S. 1991, 73: 5-96

\bibitem{A5} J. Arthur  \emph{On elliptic tempered characters.} Acta Math. 1993, 171: 73-138

\bibitem{A6}  J. Arthur \emph{ The Trace Paley Wiener theorem for Schwartz functions} Contemp.
Math. 1994, 177: 171-180

\bibitem{A7} J. Arthur \emph{On local character relations.} Selecta Math. 1996, 2: 501-579

\bibitem{A8}  J. Arthur \emph{On the transfer of distributions: weighted orbital integrals.} Duke Math. J. 1999,
99: 209-283
\bibitem{A9}  J. Arthur \emph{A stable trace formula I. General expansions.} Journal of the Inst. of Math. Jussieu 2002, 175-277

\bibitem{A10} J. Arthur \emph{A stable trace formula II. Global descent.} Invent. Math.  2001, 143:157-220


\bibitem{A11} J. Arthur \emph{A stable Trace Formula III. Proof of the Main Theorems.} Ann.  2003, 158: 769-873


\bibitem{A12} J. Arthur \emph{Problems for real groups.} Contemp. Math. 2008, 472: 39-62

\bibitem{Harish-Chandra1} Harish-Chandra \emph{Harmonic analysis on real reductive groups, I: The theory of the constant term.} J Funct Anal, 1975, 19: 104-204

\bibitem{Harish-Chandra2} Harish-Chandra  \emph{Harmonic analysis on real reductive groups, II: Wave packets in the Schwartz space.} Invent Math,
1976, 36: 1-55

\bibitem{Harish-Chandra3} Harish-Chandra   \emph{Harmonic analysis on real reductive groups, III: The Maass-Selberg relations and the Plancherel
formula.} Ann Math, 1976, 104: 117-201

\bibitem{K1}  R. E. Kottwitz   \emph{Stable trace formula: cuspidal tempered terms.} Duke Math J. 1984, 51: 611-650

\bibitem{K2} R. E. Kottwitz  \emph{Stable Trace Formula: Elliptic Singular Terms.} Math. Ann. 1986, 275:365-399

\bibitem{K3}  R. E. Kottwitz  \emph {Tamagawa numbers.} Ann. of Math. 1988, 127:3  629-646

\bibitem{K4} R. E. Kottwitz   \emph{Stable version of Arthur¡¯s formula} preprint.

\bibitem{N}  B.C. Ngo  \emph{ Le lemme fondamental pour les alg¨¨bres de Lie.} Publ. Math. Inst. Hautes ¨¦tudes Sci. No.2010, 111: 1-169

\bibitem{L1} R.P.Langlands \emph {Dimension of spaces of automorphic forms.} Proc. Sympos. Pure Math. Am.
Math. Soc. 1966, 9: 253-257

\bibitem{L2}  R.P.Langlands  \emph{On the classification of irreducible representations of real algebraic groups.} Math. Surveys Monogr. 31: Amer. Math. Soc. Providence, RI, 1989, 101-170

\bibitem{L3}  R.P.Langlands \emph{ On the functional equations satisfed by Eisenstein Series.} Lect. Notes in Math. 544:Springer-Verlag 1976, 361-500

\bibitem{LS}   R.P. Langlands, D. Shelstad
 \emph{On the definition of transfer factors} Math. Ann. 1987, 278: 219-271


 \bibitem{OW}  M. Osborne, G. Warner \emph{ Multiplicities of the integrable discrete series: The case of a non-uniform
 lattice in an R-rank one semisimple group.} J. Funct. Anal. 1978, 30: 287-310


\bibitem{S} A. Selberg \emph{Harmonic analysis and discontinuous groups in weakly symmetric spaces with applications
to Dirichlet series.} J. Indian Math. Soc.  1956, 20: 47-87


\bibitem{S1} D.Shelstad \emph{L-indistinguishability for real groups.} Math. Ann. 1982, 259: 385-430

\bibitem{S2} D.Shelstad  \emph{Tempered endoscopy for real groups I: geometric transfer with canonical factors} Contemp. Math. 2008, 472: 215-248

\bibitem{S3} D.Shelstad  \emph{Tempered endoscopy for real groups II: spectral transfer factors}  International Press,2008, 243-282

\bibitem{S4} D.Shelstad  \emph{Tempered endoscopy for real groups III: Inversion of transfer and L-packet structure.}
Represent Theory 2008, 12:  369-402


\bibitem{Sp} S. Spallone  \emph{Stable trace formula and discrete series  multiplicities.}
Pacific Joural of Mathematics.  2012, 256:  435-488

\bibitem{W1} J. L. Waldspurger \emph{Une formule des traces locale pour les algrebres de Lie p-adiques.} J. Reine
Angew. Math. 1995

\bibitem{W2}  J. L. Waldspurger \emph{Le lemme fondamental implique le transfert.} Compositio Math. 1997, 105: 153-236

\end{thebibliography}
\end{document}